\documentclass[11pt]{amsart}
\usepackage{hyperref}
\usepackage{amsfonts,amssymb,amsmath,amsthm,amscd,euscript,
array,mathrsfs}
\input{mymacros-3.sty}

\title{On equality of  ranks of local components of automorphic representations}
\author{Mohammad Bardestani}
\address{Department of Mathematics and Statistics, University of Ottawa, 585 King Edward Ave., Ottawa, ON, Canada, K1N6N5}
\email{mohammad.bardestani@gmail.com}

\author{Hadi Salmasian}
\address{Department of Mathematics and Statistics, University of Ottawa, 585 King Edward Ave., Ottawa, ON, Canada, K1N6N5}
\email{hsalmasi@uottawa.ca}

\begin{document} 

\begin{abstract}
We prove that the local components of an automorphic representation of an adelic semisimple group have equal rank in the sense of \cite{salduke}. 
Our theorem is an analogue of the results  previously obtained by Howe \cite{HoweLowRank}, Li \cite{LiDege}, Dvorsky--Sahi \cite{DvorskySahi}, and Kobayashi--Savin \cite{KobayashiSavin}. Unlike previous works which are based on explicit matrix 
realizations  and existence of  parabolic subgroups with abelian unipotent radicals, our proof works uniformly for all of the (classical as well as exceptional) groups under consideration.
Our result is an extension  of the statement known for several semisimple groups (see \cite{GanSavin}, \cite{SalManus})
that if at least one local component of an automorphic representation is a minimal representation, then all of its local components are minimal.

\end{abstract}

\maketitle

%%%%%%%%%%%%%%%%%%%%%%%%%%
\section{Introduction}
\label{SecIntr}
The notion of rank for a unitary representation of a semisimple group over a local field of characteristic zero is a  powerful tool for  studying \emph{singular} (also known as \emph{small}) unitary representations. 
The first such notion of rank, nowadays usually called \emph{$N$-rank}, was introduced by Roger Howe \cite{HoweRank}
 for the metaplectic group $\mathrm{Mp}_{2n}$. In a nutshell, Howe's idea is to consider orbits of the action of the Levi factor of the Siegel parabolic on its unipotent radical, and to associate them to unitary representations.  Howe used his $N$-rank to 
construct singular unitary representations \cite{HoweCortona}, to study the connection between singular representations and automorphic forms with degenerate Fourier coefficients \cite{HoweLowRank}, and to obtain explicit pointwise bounds for matrix coefficients of general irreducible unitary representations \cite{HoweRank}. 

Following Howe's work, a variety of notions of rank similar to Howe's $N$-rank were defined, e.g., by Scaramuzzi \cite{Scaramuzzi} for $\mathrm{GL}_n$,  by
Li \cite{LiSing}, \cite{LiDege} for classical groups, and  by Dvorsky and Sahi 
\cite{DvorskySahi}
for semisimple groups associated to Jordan algebras.
 The underlying idea of all of these works is similar to Howe's original method, namely to consider the restriction of a unitary representation to an abelian unipotent radical.

In \cite{salduke}, the second author defined a new notion of rank  that was applicable uniformly to nearly all  semisimple groups over local fields of characteristic zero, and  proved the  \emph{purity} theorem, which is one of the important steps in applications of all of the aforementioned notions of rank. The new idea that was introduced in \cite{salduke} was to define rank based on Kirillov's orbit method. 

In this article, our goal is to prove that an analogue of
a  result of Howe \cite{HoweLowRank}
on the rank of local components of automorphic representations of $\mathrm{Mp}_{2n}$ also 
holds in the context of the rank defined by the second author in \cite{salduke}. 

We now proceed to the statement of our main theorem.  
Let $\Places:=\{2,3,5,\ldots,\infty\}$ denote the set of places of $\Q$.
For every place $\pl\in\Places$,
we denote the corresponding completion of $\Q$ by $\Q_\pl$  (note that $\Q_\infty=\R$). 
Let $\bfG$ be an algebraic group that is defined over $\Q$.
%For every $\Q$-algebra $\mathbb{K}$, the group of $\mathbb{K}$-points of  $\bfG$ will be denoted by $\bfG(\mathbb K)$.
Throughout this paper, we assume that $\bfG$ satisfies the following conditions:
\begin{itemize}
\item[(i)]  $\bfG$ is connected and absolutely almost simple.
\item[(ii)] The absolute root system of $\bfG$ is not $\mathsf{A}_1$, and its highest root is defined over $\Q$.
\item[(iii)] The weak approximation property holds for $\bfG$ 
with respect to every place $\nu\in\Places$. In other words, $\bfG(\Q)$ is dense in $\bfG(\Q_\pl)$ for every place $\pl\in\Places$. 
\end{itemize}
% We need to impose a mild technical condition on $\bfG$, given in Eq. \eqref{condition-gp} below. 

Let $\A:=\prod'_{\pl\in\Places}\Q_\pl$ denote the ring of adeles
(where $\prod'$ indicates the restricted product). 
  We  consider a finite topological central covering 
\begin{equation}
\label{exactseqA}
1\to F\xrightarrow{\ \mathsf{i}\ } G_\A
\xrightarrow{\ \mathsf{p}\ }
 \bfG(\A)\to 1,
\end{equation}
which is split over  $\bfG(\Q)$. We fix a splitting section for 
$\bfG(\Q)$, and we
denote the image of $\bfG(\Q)$ under this section  by  $G_\Q\subseteq G_\A$. 
Now let $(\pi,\ccH)$
be an irreducible unitary representation  of $G_\A$. For every place $\pl\in\{2,3,5,\ldots,\infty\}$ of $\Q$, we denote the local component of
$\pi$ at place $\pl$ by  $\pi_\pl$
(for a precise definition of local components for covering groups, see
Remark \ref{rmk-loccpnt}).
Then $\pi_\pl$ is an irreducible unitary representation of a finite topological central extension $G_{\pl}$ of $\bfG(\Q_\pl)$.
Let $\mathrm{rk}(\pi_\pl)$ denote the rank of $\pi_\pl$ according to \cite{salduke} (see Definition \ref{dfnnurkk} below). Our main theorem is the following (see Theorem \ref{THMMaIn} for a restatement and proof).
\begin{theorem*}
 Let $G_\A$ be a finite topological central extension of $\bfG(\A)$, with $\bfG$ as above. Assume that $G_\A$ splits over 
$\bfG(\Q)$, and let $G_\Q\sseq G_\A$ be the subgroup defined 
above.
 Let $(\pi,\ccH)$ be an irreducible unitary representation of $G_\A$ which occurs as a subrepresentation of $L^2(G_\Q\bls G_\A)$.
Then $\mathrm{rk}(\pi_\pl)$ is independent of
$\pl\in\Places$.
\end{theorem*}

For further information on basic properties of the central extension 
\eqref{exactseqA}, see for example \cite[Sec. I.1.1]{MW}. Condition (iii) on $\bfG$ is  known to hold for many general classes of groups, e.g., simply connected or split groups.

We briefly comment on the relation between our paper and the recent article of Kobayashi and Savin \cite{KobayashiSavin}. In \cite[Thm 1.1]{KobayashiSavin}, Kobayashi and Savin prove a similar statement for groups that possess a maximal parabolic subgroup which has an abelian unipotent radical and which is conjugate to its opposite. Using the methods and results of \cite[Sec. 6]{salduke} on the relation between the $N$-rank and the rank introduced by the second author, one should be able to deduce the result of Kobayashi and Savin 
from our theorem for unitary automorphic representations. 
%Furthermore, Kobayashi and Savin assume that all of the local components of the automorphic representation under consideration are small. Our theorem does not need such an assumption.
In addition, our method of proof is different from the one used by Kobayashi and Savin, because our techniques rely heavily on Kirillov's orbit method for unipotent groups.

We will now explain the key ideas of the proof of our main theorem. The proof is inspired by the original method of  Howe, which is to  detect the rank of a unitary representation by means of operators that come from elements of the convolution algebra of Schwartz (or locally constant compactly supported) functions on a unipotent group $U_{\pl}\subseteq G_{\pl}$. As shown by Howe, this is relatively easy if $U_{\pl}$ is abelian, as one only needs to choose a function whose Fourier transform has a suitable support (see the proof of \cite[Lem. 2.4]{HoweLowRank} and \cite[Lem. 3.2]{LiDege}).
However, in our case $U_{\pl}$ is non-abelian, and finding an element of the Schwartz algebra of $U_\pl$ which separates a point outside
a closed subset 
of the unitary dual of $U_\pl$
requires a nontrivial argument (see 
Section \ref{sec-existe} below).  Of course this can be done by an element of the $C^*$-algebra of $U_{\pl}$, but there is no reason to expect that such an element should come from the Schwartz algebra.
 It is proved in \cite{LudwigArk} that the primitive ideal of  every irreducible unitary representation
 of 
a nilpotent Lie group has a dense intersection with the algebra of Schwartz functions, but this is not enough for our argument, and we need a generalization of this statement for annihilators of many closed subsets of the unitary dual of $U_{\pl}$. We do not know  
if such a  statement is true. In the archimedean case, we get around this technical issue by using classical results of Dixmier 
\cite{DixORF}
and Hulanicki (see \cite{LRS} for a detailed discussion)
about functional calculus of nilpotent Lie groups \cite{DixORF}. 
In the non-archimedean case, we address the analogous technical issue 
 in Sections \ref{FuncCalcSec}--\ref{sec-existe}
using an interesting  result of
S. Gelfand and Kazhdan \cite{GelfandKazhdan}
 about closed subsets of the unitary dual of a $p$-adic unipotent group.

Another point of diversion of our proof from Howe's method is that, if $U_\A$ is 
non-abelian then it is 
not a Type I group \cite[Sec.~7]{Moore}, and therefore uniqueness of direct integral decomposition does not hold for its unitary representations. We circumvent this issue by
reducing the problem to the decomposition of 
$L^2(U_\Q\setminus U_\A)$. By a result of Moore \cite[Thm 11]{Moore}, the latter representation decomposes as a multiplicity-free direct sum of irreducible unitary representations that can be constructed using Kirillov's orbit method (see Section \ref{KirillovA}). Of course we cannot use restriction to $U_\A$ as a $U_\A$-intertwining map 
from $L^2(G_\Q\setminus G_\A)$ to $L^2(U_\Q\setminus U_\A)$, but as shown in the proof of Theorem \ref{THMMaIn}, we can carefully use the restriction of smooth vectors.

Even though our theorem is stated only for groups over the  adele ring of $\Q$, we believe that it remains true for groups over the adele ring of an algebraic number field. The reason why we have hesitated to state our theorem  in the latter more general context is that the results of Moore in \cite{Moore} on the decomposition of $L^2(U_\Q\bls U_\A)$ are only stated for the adele ring of $\Q$.
%Kirillov theory over p-adics is explained in Section 6 of Dijk's paper.
  It is quite likely that Moore's result holds in the general case, and therefore the proof of our theorem applies \emph{mutatis mutandis} to groups over the adele ring of any algebraic number field.

Finally, 
we remark that it would be interesting to see if our main theorem can be used to obtain interesting information about the wavefront sets of the local components of a small automorphic representation of an exceptional group $\bfG$ (see \cite{SahiGourevich} and 
\cite{HeHong}
for results in this direction).
 We intend to come back to this problem in the near future.\\

\noindent\textbf{Acknowledgement.} We thank
Jean Ludwig for helping us with the proof of Proposition \ref{prp-funcal},  Alireza Salehi Golsefidy for responding to our questions about approximation in algebraic groups, and  Wee-Teck Gan for  pointing to 
\cite{MW}
in relation to Remark \ref{rmk-loccpnt}. We also thank Gerrit van Dijk, Roger Howe, Siddhartha Sahi, and  Martin Weissman for helpful and encouraging discussions and correspondences.
The second author acknowledges support from a Discovery Grant by the Natural Sciences and Engineering Research Council of Canada.

\section{The rank parabolic}
\label{therankparabolic}

This section is devoted to the notation and structure theory that we will need to be able to recall  
 the notion of  rank  and the purity theorem proved in \cite{salduke}. The definition of rank, which will be reviewed in Section
\ref{Sec-nurknk} below, 
  is based  on the existence of \emph{Heisenberg parabolic subgroups}, which we will define first.

Recall that $\bfG$ is an
algebraic group defined over $\Q$, which satisfies conditions (i)--(iii) of  Section \ref{SecIntr}.  Fix a maximally $\Q$-split Cartan subgroup $\bfH\sseq \bfG$. Let $\g h$ denote the Lie algebra of $\bfH$, and let $\Phi\subseteq\g h^*$ denote the absolute root system of $\bfG$ corresponding to $\bfH$. 
%Fix a maximal $\Q$-split torus $\bfA\subseteq \bfH$,  and let $\Sigma:=\left\{\alpha|_{\g a}^{}\,:\,\alpha\in\Delta\right\}$ be the restricted root system of $\bfG$, where $\g a:=\Lie(\bfA)$. 
%Fix a minimal $\Q$-parabolic subgroup of $\bfG$ containing $\bfH$ and choose a compatible 
Choose a positive system $\Phi^+$  in $\Phi$, and let $\Delta$ denote the corresponding base of $\Phi$.
%Let $\prec$ denote the standard partial order on $\Phi$, that is, for every $\alpha,\beta\in\Phi$ we set $\alpha\prec\beta$ if and only if $\beta-\alpha$ is in the semigroup generated by $\Delta$. 
%Let $\kappa:\g g\times \g g\to\C$ denote the Killing form of $\g g:=\Lie(\bfG)$. For every $\alpha,\beta\in\Phi$, we set $ (\alpha,\beta):=\kappa(h_\alpha,h_\beta)$,where $h_\alpha,h_\beta\in\g h$ are the coroots corresponding to $\alpha,\beta\in \Phi$. Set
%$\langle\alpha,\beta\rangle:=2\frac{(\alpha,\beta)}{(\beta,\beta)}$.
We denote the highest root of $\Phi$ by
 $\beta_\Phi$. Thus, condition (ii) of Section 
 \ref{SecIntr} states that
\begin{equation}
\label{condition-gp}
\Phi\neq\mathsf A_1\text{ and }
 \beta_\Phi\text{ is defined over }\Q.
 \end{equation}
%, that is, $h_{\beta_\Phi}\in\g a$. 
Let $(\cdot,\cdot)$ denote the canonical symmetric bilinear form induced on $\g h^*$ by the Killing form, and set
\[
\Sigma_\Phi:=\{\alpha\in\Delta\ :\ (\alpha,\beta_\Phi)\neq 0\}.
\]
Then 
$
|\Sigma_\Phi|=2
$ if $\Phi$ is of type $\mathsf{A}_n$, $n>1$, and 
$
|\Sigma_\Phi|=1
$
otherwise \cite{Torasso}.
Furthermore, the unipotent radical 
$\bfU_{\Delta- \Sigma_\Phi}$ of the $\Q$-parabolic subgroup
$\bfP_{\Delta- \Sigma_\Phi}=\bfL_{\Delta- \Sigma_\Phi}\ltimes \bfU_{\Delta- \Sigma_\Phi}$  of $\bfG$ is a Heisenberg group
%generated by $\{\bfU_\alpha\,:\,\alpha\in \widetilde\Sigma_\Phi\}$, where
%\[\widetilde\Sigma_\Phi:=\Phi^+\cup\{-\alpha\,:\,\alpha\in\Phi^+\text{ and }\lag \alpha,\beta\rag=0\text{ for every }\beta\in\Sigma_\Phi\}.\]
with center
 $\bfZ_{\Delta- \Sigma_\Phi}=\bfU_{\beta_\Phi}$,
where $\bfU_{\beta_\Phi}$ is the unipotent subgroup of $\bfG$ corresponding to $\beta_\Phi$ (see \cite{Torasso} or 
 \cite{GanSavin} for details).
We call $\bfP_{\Delta-\Sigma_\Phi}$ the 
\emph{Heisenberg parabolic subgroup} associated to $\Phi$.

%From now on we assume $\bfG\neq \mathrm{SL}_2$.
Next, starting with $\Phi_1:=\Phi$, we define a chain of irreducible root systems  \begin{equation*}
\Phi=\Phi_1\supsetneq\cdots\supsetneq 
\Phi_r \supsetneq\Phi_{r+1}=\varnothing,
\end{equation*}
 by the following inductive construction.
Assume that $\Phi_n$ has been defined 
for some $n\geq 1$, and set $\Phi_n':=\{\alpha\in\Phi_n\, :\, (\alpha,\beta_{\Phi_n})=0\}$.
Observe that $\Phi_n'$ is the root system of the Levi factor of the  
Heisenberg parabolic subgroup associated to $\Phi_n$. By examination of all Dynkin diagrams, it follows that 
$\Phi_n'$ is a Cartesian product of at most three irreducible subsystems. Furthermore, at most one of these subsystems satisfies the assumption \eqref{condition-gp}. We set $\Phi_{n+1}$ equal to the subsystem that satisfies \eqref{condition-gp}
if it exists (in this case it will always be unique), and set $\Phi_{n+1}:=\varnothing$ otherwise.
%has at most one irreducible subsystem is either irreducible or of type $\mathsf A_1\times\mathsf R$ (but $\mathsf R$ might also be reducible). In the latter case, we denote a subsystem of $\Phi_n$ of type $\mathsf R$ by $\Phi_n''$ (and the choice of $\Phi_n''$ does not really matter.) We now set
%\[\Phi_{n+1}:=\begin{cases}\Phi_n'&\text{ if }\Phi_n'\text{ is irreducible, }\mathrm{rank}(\Phi_n')>1,\text{ and }\beta_{\Phi_n'}\text{ is defined over }\Q,\\ \Phi_n''&\text{ if }\Phi_n'\cong\mathsf A_1\times \mathsf R,\text{ where }\mathsf R\text{ is irreducible,   }\mathrm{rank}(\mathsf R)>1,\text{ and }\beta_{\mathsf R}\text{ is defined over }\Q,\\ \varnothing&\text{ otherwise}.\end{cases}\]
If $\Phi_{n+1}\neq\varnothing$, then we set  $\Phi_{n+1}^+=\Phi_{n+1}\cap \Phi^+$. The inductive construction stops as soon as $\Phi_{n+1}=\varnothing$.

Set $\Gamma:=\bigcup_{i=1}^r\Sigma_{\Phi_i}$. From now on let
$\bfP=\bfL\ltimes \bfU$ denote the standard $\Q$-parabolic subgroup of $\bfG$ corresponding to $\Delta-\Gamma$. 
%By this we mean that the root system of $\bfL_\Gamma$ is spanned by $\Delta-\Gamma$.
Note that
\begin{equation}
\label{ngammatower}
\bfU\simeq \bfU^{}_{r}\ltimes \left(\bfU^{}_{r-1}\ltimes \left(\cdots\ltimes \bfU^{}_{1}\right)\cdots\right),
\end{equation}
where $\bfU_d$ for $1\leq d\leq r$ denotes the unipotent radical of the Heisenberg parabolic subgroup associated to $\Phi_d$.
In particular, $\bfU$ is a tower of semidirect products of Heisenberg groups. From now on, we denote the length of the tower of semidirect products \eqref{ngammatower} by $r(\bfG)$.
%where $\bfP_{S_{\Phi_n}}=\bfL_{S_{\Phi_n}}\ltimes \bfU_{S_{\Phi_n}}$ is the Heisenberg parabolic subgroup of the reductive subgroup of $\bfG$ that corresponds to $\Phi_n$.

\begin{ex}
Assume that $\bfG=\mathrm{SL}_n$. If $n=2k+1$, $k\geq 1$, then $\bfP$ is the Borel subgroup, whereas if $n=2k$, $k>1$, then $\bfP$ is the parabolic subgroup corresponding to the middle node of the Dynkin diagram. 
\end{ex}
The following table indicates the values of $r(\bfG)$ when the algebraic group $\bfG$ is $\Q$-split.

\begin{center}
\renewcommand{\arraystretch}{1.25}
\begin{tabular}{c|c}
$\bfG$ & $r(\bfG)$\\
\hline
$\mathsf{A}_n$ & $\lfloor \frac{n}{2}\rfloor $ \\
$\mathsf{B}_n$ & $\lfloor \frac{n}{2}\rfloor $ \\
$\mathsf{C}_n$ & $ n-1 $ \\
$\mathsf{D}_n$ & $\lfloor \frac{n-1}{2}\rfloor $ \\
$\mathsf{E}_n$ & $\lfloor \frac{n}{2}\rfloor $ \\
$\mathsf{F}_4$ & $3 $ \\
$\mathsf{G}_2$ & $1$ \\
\end{tabular}
\end{center}

\section{The $\pl$-rank of a unitary representation}
\label{Sec-nurknk}
For every place $\pl\in\Places$, we 
identify  $\bfG(\Q_\pl)$ with a subgroup of 
$ \bfG(\A)$
via the canonical embedding
$\Q_\pl\into \A$.
The exact sequence
\eqref{exactseqA} restricts to 
a finite topological central extension 
\begin{equation}
\label{11FFpll}
1\to F_\pl\xrightarrow{\ \mathsf{i}_\pl\ }G_\pl\xrightarrow{\ \mathsf{p}_\pl^{}\ }
\bfG(\Q_\pl)\to 1,
\end{equation}
where $G_\pl:=\mathsf{p}^{-1}(\bfG(\Q_\pl))
$ and $F_\pl:=\mathsf{i}^{-1}(G_\pl)$.
%be a finite topological 
%\footnote{It is a question of G. Prasad whether every finite central extension of $\bfG(\Q_\pl)$ is topological. This is known for quasi-split groups \cite{Sury}.}
%central extension of $\bfG(\Q_\pl)$.
% the group of $\Q_\pl$-points of $\bfG$.  
 As shown for example in \cite[Lem.~II.11]{DufloTh} or \cite[Appendix I]{MW}, the above central extension splits over $\bfU(\Q_\pl)$, and there is a unique splitting section
\begin{equation}
\label{secnu}
\mathsf{s}_\pl:\bfU(\Q_\pl)\to G_\pl.
\end{equation}
Set $U_\pl:=\mathsf{s}_\pl(\bfU(\Q_\pl))$. %Then $U_\pl$ and  $\bfU(\Q_\pl)$ are isomorphic topological groups. 
%UNIQUENESS FOLLOWS FROM DIVISIBILITY OF UNIPOTENT GROUPS

Recall that $r=r(\bfG)$.
For $1\leq n\leq r$, let $\bfZ_n$ denote the center of $\bfU_n$. Set
$
U_{n,\pl}:=\mathsf{s}_\pl(\bfU_{n}(\Q_\pl))
$
and 
$ 
Z_{n,\pl}:=\mathsf{s}_\pl(\bfZ_{n}(\Q_\pl))
$.
 From \eqref{ngammatower} it follows that
\begin{equation}
\label{ngammatowerlocal}
U_{\pl}\simeq U_{r,\pl}\ltimes (U_{r-1,\pl}\ltimes (\cdots\ltimes U_{1,\pl})\cdots).
\end{equation}
Furthermore, $U_{n,\pl}$ is a Heisenberg group, that is, $W_{n,\pl}:=U_{n,\pl}/Z_{n,\pl}$ is  a symplectic $\Q_\pl$-vector space, so that the symplectic group $\mathrm{Sp}(W_{n,\pl})$ 
acts on $U_{n,\pl}$ by automorphisms which fix 
$Z_{n,\pl}$ pointwise.
As shown in 
\cite[Sec. 5]{salduke},
for every $1\leq n\leq r-1$, the action 
by conjugation of 
\[
U^{(n)}_{\pl}:=U_{r,\pl}\ltimes (U_{r-1,\pl}\ltimes (\cdots\ltimes U_{n+1,\pl})\cdots)
\]
on $U_{n,\pl}$  factors through $\mathrm{Sp}(W_{n,\pl})$. Furthermore, again by \cite[Lem.~II.11]{DufloTh}, the  
metaplectic central extension
\begin{equation}
\label{lem-oscil}
1\rightarrow \Z/2\Z\rightarrow{\mathrm{Mp}}(W_{n,\pl})
\rightarrow
\mathrm{Sp}(W_{n,\pl})
\rightarrow 1
\end{equation}
splits over $U^{(n)}_\pl$.
By the Stone--von Neumann Theorem, for every nontrivial unitary character $\chi_{n,\pl}$ of $Z_{n,\pl}$ there exists a unique  irreducible unitary representation 
$\rho_{\chi_{n,\pl}^{}}$
of $U_{n,\pl}$ with central character $\chi_{n,\pl}$. 
Next we extend $\rho_{\chi_{n,\pl}}$ to a unitary representation 
 of $U_{\pl}$. To this end,
 first we  extend $\rho_{\chi_{n,\pl}^{}}$ to $U^{(n)}_\pl\ltimes U_{n,\pl}$
by means of the restriction of the
oscillator representation of
${\mathrm{Mp}}(W_{n,\pl})\ltimes U_{n,\pl}$
(this can be done because the exact sequence \eqref{lem-oscil} is split over $U^{(n)}_\pl$).
Subsequently, we
extend 
$\rho_{\chi_{n,\pl}^{}}$
to
 $U_{\pl}$ via the canonical quotient map 
 \[
 U_{\pl}\to U_{\pl}/(U_{n-1,\pl}\ltimes (U_{n-2,\pl}\ltimes (\cdots\ltimes U_{1,\pl})\cdots))\cong U^{(n)}_\pl\ltimes U_{n,\pl}.
 \]
As shown in \cite[Cor. 4.2.3]{salduke}, 
for $1\leq d\leq r$,
the tensor product 
$\rho_{\chi_{1,\pl}^{}}\otimes \cdots\otimes \rho_{\chi_{d,\pl}^{}}$ is an irreducible unitary representation of $U_{\pl}$.
\begin{dfn}
\label{DfnU(d)}
For every $1\leq d\leq r(\bfG)$, the irreducible unitary representations
$\rho_{\chi_{1,\pl}^{}}\otimes \cdots\otimes \rho_{\chi_{d,\pl}^{}}$ of 
$U_{\pl}$ that are constructed above are called \emph{rankable} of rank $d$. The set of 
unitary equivalence classes of rankable representations of $U_\pl$ of rank $d$  will be denoted by $\widehat{U}_\pl(d)$.
\end{dfn}

Let $\widehat{U}_\pl$ denote the unitary dual of $U_\pl$.
Recall that from the well known results of Kirillov in the archimedean case and \cite[Thm 3]{Moore} in the non-archimedean case, the orbit method gives a bijection between 
$\widehat{U}_\pl$ and
the coadjoint orbits of $U_\pl$. 

\begin{lem}
\label{dimcoadj}
For every $1\leq d\leq r(\bfG)$, the coadjoint orbit corresponding to any rankable representation of $U_{\pl}$ of rank $d$ is an analytic manifold  of dimension $\dim(U_{1,\pl})+\cdots+\dim(U_{d,\pl})-d$.
\end{lem}

\begin{proof}
This is basically a restatement of 
\cite[Cor. 4.2.3]{salduke}.
For the definition of an analytic manifold over a local field,  see
\cite[Sec. II.3.2]{Serre}. 
 Unfortunately, in \cite{salduke}, the details of the proof of the fact that  coadjoint orbits are analytic manifolds were not given.
The latter statement is a consequence of the following general observation.
Let $\bfU$ be a unipotent $\Q_\pl$-algebraic group, and let  $\bfU\times \mathbf V\to \mathbf V$ be a 
$\Q_\pl$-action of 
 $\bfU$ on an affine space $\mathbf V$. Let 
 $V_\pl$ and $U_\pl$ denote the sets of 
 $\Q_\pl$-points of $\mathbf U$ and $\mathbf V$. Let $x\in V_\pl$, and let $\oline {\mathcal  O}_x\subseteq\mathbf V$ denote the $\bfU$-orbit of $x$. Set $\mathcal O_x:=\{u\cdot x\,:\, u\in U_\pl\}$.
From \cite[Lem. 7.1]{Moore} it follows that $\mathcal O_x=\oline{\mathcal O}_x\cap V_\pl$, and from 
\cite[Prop. 4.10]{BorelLAG} it follows that
$\oline{\mathcal O}_x$ is a Zariski-closed subset of $\mathbf V$.
%A corrected proof is in a note by Brian Conrad.
Consequently, $\mathcal O_x$ is a Zariski-closed subset of $V_\pl$, hence also a closed subset of the analytic $\Q_\pl$-manifold $V_\pl$.
Finally, from \cite[II.4.5, Thm 3]{Serre} and 
\cite[Sec. II.4.5, Thm 4]{Serre}
it follows that $\mathcal O_x$ is also an analytic $\Q_\pl$-submanifold of $V_\pl$.
\end{proof}

Now let $\pi_\pl$ be a unitary representation of $G_\pl$. We can express the
restriction of $\pi_\pl$ to $U_{\pl}$ in an essentially unique way as a direct integral
\begin{equation}
\label{piNGamma}
\pi_\pl\big|_{{U}_{\pl}}^{}=
\int^\oplus_{{\widehat{U}_\pl}}n_\sigma\sigma \,d\mu(\sigma)\ \text{ where }\ n_\sigma\in\N\cup\{\infty\}.
\end{equation} 
As a consequence of \cite[Thm 5.3.1]{salduke}, we have the following theorem (see Remark \ref{rmkPFSDD} below).
\begin{thm}
\label{mainsalduke}
Let $\pi_\pl$ be a nontrivial irreducible unitary representation of $G_\pl$. Then there exists a unique integer $1\leq d=d(\pi_\pl)\leq r(\bfG)$   such that $\mu\left(\widehat{U}_\pl- \widehat{U}_\pl(d)\right)=0$, where $\mu$ is the Borel measure in the decomposition \eqref{piNGamma}.
\end{thm}

\begin{dfn}
\label{dfnnurkk}
The integer $d(\pi_\pl)$ that is associated to the unitary representation $\pi_\pl$ in 
Theorem \ref{mainsalduke} is called the
\emph{$\pl$-rank} of  $\pi_\pl$.
\end{dfn}
\begin{rmk}
\label{rmkPFSDD}
The proof of  Theorem \ref{mainsalduke} needs some clarification.
The tower of semidirect products
of Heisenberg groups that appears in  
  \cite{salduke} is defined 
by successive consideration of highest roots which are defined over the local field $\Q_\pl$. 
Obviously, a highest root that is defined over $\Q_\pl$ is not necessarily defined over $\Q$. 
  Therefore the tower of semidirect products in \cite{salduke}
  might be longer than 
the one introduced in \eqref{ngammatowerlocal}. 
More precisely, it will be a unipotent group of the form
\[
U^\flat:=U_{r',\pl}\ltimes (\ldots\ltimes U_{r,\pl}\ltimes (U_{r-1,\pl}\ltimes (\cdots\ltimes U_{1,\pl})\cdots)),
\]
where $r'\geq r$.
% This is related to the fact that  the split $\Q_\nu$-rank of the Cartan subgroup $\bfH$ might be strictly bigger than its split $\Q$-rank. 
However, Theorem \ref{mainsalduke} 
follows from \cite[Thm 5.3.1]{salduke} and the fact 
that a rankable representation $\rho$ of $U^\flat$ 
(in the sense defined in \cite{salduke})
of rank $d$  restricts to a rankable representation  of $U_\nu$ of rank  $\min\{d,r\}$. %In particular, the rank of the restriction of $\rho$ to $U_\nu$ is uniquely determined by the rank of $\rho$.
\end{rmk}

\begin{rmk}
In \cite[Sec. 1]{salduke}, it is assumed that the residual characteristic of the local field is odd. This assumption is superfluous and indeed it is never used in the proofs of \cite{salduke}. 
\end{rmk}

\begin{rmk}
The relation between the notion of $\nu$-rank defined in Definition \ref{dfnnurkk} and the rank in the sense of Howe \cite{HoweRank}, 
Li \cite{LiSing}, and Scaramuzzi \cite{Scaramuzzi} 
was investigated in detail  in \cite[Sec. 6]{salduke}.
\end{rmk}
\section{Smooth forms of moderate growth}
%where $\bfG(\A)$ is the group of $\A$-points of $\bfG$
%. (see \cite[Sec. I.1.1]{MW})
%Recall that we assume that the central extension splits over $\bfG(\Q)$, and we fix a splitting section for $\bfG(\Q)$, whose image is denoted by  $G_\Q\subseteq G_\A$.

%\label{Gpl=fpp}
For every place $\pl\in\Places$, we set 
$\bfG^{\pl}:=
\prod'_{\eta\in\Places-\{\pl\}}\bfG(\Q_{\eta})
$, and we identify 
$\bfG^{\pl}$ with a subgroup of $\bfG(\A)$ in a natural way.
Furthermore, we set 
\begin{equation}
\label{Gupv}
G^\nu:=\mathsf{p}^{-1}(\bfG^\nu)
.
\end{equation}

According to
\cite[Appendix I]{MW}, the central extension 
\eqref{exactseqA}
 splits over $\bfU(\A)$, and the splitting section
$\mathsf{s}:\bfU(\A)\to G_\A$ is unique. 
Set 
$P_\pl:=\mathsf{p}^{-1}(\bfP(\Q_\pl))$, where we identify $\bfP(\Q_\pl)$ with a subgroup of $\bfG(\Q_\pl)\sseq \bfG(\A)$.
Uniqueness of the sections $\mathsf{s}$ and $\mathsf{s}_\pl$, defined in \eqref{secnu}, implies 
$\mathsf{s}\big|_{\bfU(\Q_\pl)}=\mathsf{s}_\pl$, hence
\begin{equation}
\label{secglobal}
U_\pl=\mathsf{s}(\bfU(\Q_\pl))
\ \text{ for every }\ \pl\in\Places,
\end{equation}
and $U_\pl$ is normalized by $P_\pl$.
%
%BECAUSE FOR g\in G, the map u -> g^{-1}s(p(g)up(g)^{-1}) is also a section, hence it should be equal to s(u). Now this implies tha gs(u)g^{-1} belongs to s(U_\pl).
%
%
%Furtheremore, the central extension \eqref{exactseqA}restricts to a central extension similar to \eqref{11FFpll} forthe group $G_\pl$ that is defined in \eqref{Gpl=fpp}.
Set $\A_{\mathrm{fin}}:=
\prod'_{\pl\in\Places-\{\infty\}}\Q_\pl$, so that $\A\cong \R\times \A_{\mathrm{fin}}$. Furthermore, set
% if  the splitting section $\mathsf{s}_\pl$ of 
%\eqref{secnu} 
%is chosen to be the restriction of 
%the section $\mathsf{s}:\bfU(\A)\to G_\A$. 
\[
U_\A:=\mathsf{s}(\bfU(\A)),\ U_\Q:=G_\Q\cap U_\A,\
\text{and } 
G_{\A_{\mathrm{fin}}}
:=\mathsf{p}^{-1}(\bfG(\A_\mathrm{fin})
).
\] 
%where$\bfG(\A_\mathrm{fin})\subseteq\bfG(\A)$ denotes the group of  $\A_{\mathrm{fin}}$-points of $\bfG$.
%Uniqueness of the splitting sections $\mathsf{s}$ and $\mathsf{s}_\pl$ 
%implies that $\mathsf{s}$ and $\mathsf{s}_\pl$ are identical on $\bfU(\Q_\pl)$, hence
%the group $U_\pl$ defined in \eqref{secglobal} is identical to the one defined in Section \ref{therankparabolic}.It also 
%implies that  $U_\pl$ (respectively, $U_\A$) is normalized by $P_\pl$ (respectively, $P_\A$).

%Finally, for every  $\mathsf{S}\subseteq\Places$, we identify $G_\Q$ with its image inside $\prod'_{\pl\in\mathsf{S}}G_\pl$ under the natural projection $\mathsf{s}()$.

 From now on, we fix a norm $\|\cdot\|$ on $G_\R$ 
as follows. We choose a  representation $\iota:G_\R\to \mathrm{SL}_n(\R)$ for some $n>1$ which descends to a faithful representation of $\bfG(\R)$ whose image is closed in $\mathrm{Mat}_{n\times n}(\R)$,
and define \[
\|g\|:=
\left(\sum_{1\leq i,j\leq n}|x_{i,j}|^2\right)^\frac{1}{2}
\ \text{ where }\ {\iota}(g)=[x_{i,j}]^{}_{1\leq i,j\leq n}.
\]
The canonical injection $U_\Q\bls U_\A\into G_\Q\bls G_\A$ is a homeomorphism of $U_\Q\bls U_\A$ onto a closed subset of $G_\Q\bls G_\A$.  The homogeneous spaces $G_\Q\bls G_\A$ and $U_\Q\bls U_\A$ have finite invariant measures (in fact 
$U_\Q\bls U_\A$ is compact).
As usual, 
$L^2(G_\Q\bls G_\A)$ denotes 
the space of complex valued square integrable functions on $G_\Q\bls G_\A$.
The representation of $G_\A$ on $L^2(G_\Q\bls G_\A)$ by right translation will be denoted by $
\sfR(\cdot)$, that is,\[
\big(\sfR(g)f\big)(x):=f(xg)\text{ for }
f\in L^2(G_\Q\bls G_\A),\,x\in G_\Q\bls G_\A\text{ and }g\in G_\A. 
\]
%For every place $\pl$, the restriction of the right regular representation of $G_\A$ to the local component $G_\nu\subseteq G_\A$ is a strongly continuous unitary representation of $G_\nu$ on $L^2(G_\Q\bls G_\A)$ which will be denoted by $\sfR_\nu$. 
An irreducible unitary representation of $G_\A$ is called an \emph{automorphic representation} if it occurs as a subrepresentation of $\left(\sfR,L^2(G_\Q\bls G_\A)\right)$.

%\begin{rmk}For the next definition, recall that any element  $f\in L^2(G_\Q\bls G_\A)$ can be represented  by square-integrable functions $\oline{h}:G_\Q\bls G_\A\to \C$, and any two such functions agree on the complement of a null subset of $G_\Q\bls G_\A$. Let $\psf:G_\A\to G_\Q\bls G_\A$ be the canonical projection. Then $\oline{h}=h\circ \psf$for a measurable function $h:G_\A\to \C$ which is left $G_\Q$-invariant. Furthermore, $h$ is uniquely determined by $f$ up to a null subset of $G_\A$ with respect to the Haar measure (this follows from \cite[Lem. 1.3]{MackeyIndI}, which states that a Borel set $\Omega\subseteq G_\Q\bls G_\A$ has measure zero if and only if $\psf^{-1}(\Omega)\subseteq G_\A$ has Haar measure zero).The equivalence class of left $G_\Q$-invariant complex-valued functions on $G_\A$ which agree with $h$ up to a null subset of $G_\A$  will be denoted by $[h]\in L^2(G_\Q\bls G_\A)$. \end{rmk}

\begin{rmk}
\label{dfnabus}
From now on, we identify functions $f:G_\Q\bls G_\A\to \C$ with left $G_\Q$-invariant functions $f:G_\A\to\C$ in the obvious way. 
Any two left $G_\Q$-invariant measurable maps $G_\A\to\C$ that descend to maps $G_\Q\bls G_\A\to \C$ which are almost everywhere equal to $f$, are also equal everywhere except on a subset of $G_\A$ of Haar measure zero
(this follows from \cite[Lem. 1.3]{MackeyIndI}).
\end{rmk}

\begin{dfn}
\label{dfnAuFo}
A continuous function $f:G_\Q\bls G_\A\to\C$ is called \emph{smooth} if it  satisfies the following two conditions.
\begin{itemize}
%\item[(i)] $f\in L^2(G_\Q\bls G_\A)$. 
\item[(i)] There exists a compact open subgroup $K\subseteq G_{\A_\mathrm{fin}}$ such that $f(gk)=f(g)$ for $g\in G_\A$ and $k\in K$.
\item[(ii)]
For every $x\in G_\A$, the map 
$G_\R\to\C$, $y\mapsto f(xy)$ is in $C^\infty(G_\R)$. 
\end{itemize}
A continuous function $f:G_\Q\bls G_\A\to\C$ is said to be of 
\emph{moderate growth} if it is smooth and satisfies
\begin{equation}
\label{mfboundeq}
|f(xy)|\leq c_{x,f}\|y\|^{m_f}
\text{ for  every }
x\in G_\A
\text{ and every }y\in G_\R,
\end{equation} where
$m_f\in\R^+$ depends only on $f$,  and $c_{x,f}\in\R^+$ depends only on $x$ and $f$.
\end{dfn}
Let $C_c^\infty(G_\A)$ denote the space of smooth compactly supported functions on
 $G_\A$ 
 \cite[Lem. I.2.5]{MW}.
\begin{rmk}
\label{rmkAuFo}
Let $(\pi,\ccH)$ be a unitary representation of 
$G_\A$. 
Fix a Haar measure $dg$ on $G_\A$ and set
\[
\ccH^\circ:=\big\{\pi(\phi)v\,:\, v\in\ccH\ \text{ and }\ \phi\in C_c^\infty(G_\A)\big\},
\text{ where }
\pi(\phi)v:=\int_{G_\A}\phi(g)\pi(g)v\, dg.
\]
Note that $\ccH^\circ$ is a $G_\A$-invariant dense 
subspace of $\ccH$.
We call $\ccH^\circ$ the \emph{G\aa rding space} of $(\pi,\ccH)$.
If $(\pi,\ccH)$ is a 
 subrepresentation of  $L^2(G_\Q\bls G_\A)$, then from
\cite[Lem. I.2.5]{MW} it follows that
every element of $\ccH^\circ$ can be represented by a unique smooth and moderate growth map $G_\Q\bls G_\A\to \C$.
\end{rmk}

\begin{rmk}
\label{rmk-loccpnt}
Consider an irreducible unitary representation  $(\pi,\ccH)$ of $G_\A$. If $G_\A=\bfG(\A)$, then as is well known, we can express $\pi$ as a restricted tensor product $\otimes_{\pi\in\Places}'\pi_\pl$, where each $\pi_\pl$ is an irreducible unitary representation of $G_\pl=\bfG(\Q_\pl)$.  The $\pi_\pl$ are called the \emph{local components} of $\pi$. If $G_\A\neq \bfG(\A)$, then $G_\A$ is not a restricted product of the local factors $G_\pl$, and therefore the above definition of local components is not totally valid. There are various ways to fix this issue by generalizing the notion of local components (and possibly the restricted tensor product) to representations of $G_\A$. The easiest way, which is sufficient for our goals, is as follows. For every $\pl_\circ\in\Places$, the group $G_\A$ is an almost direct product of the groups  $G_{\pl_\circ}$ and
$G^{\pl_\circ}$ which are defined in 
Section \ref{Sec-nurknk} and \eqref{Gupv}.
% $\mathsf{p}^{-1}(\bfG^{\pl_\circ})$,where $\bfG^{\pl_\circ}:=\prod'_{\pl\neq\pl_\circ}\bfG(\Q_{\pl})\sseq \bfG(\A)$, and $\mathsf{p}:G_\A\to\bfG(\A)$ is the projection map defined in \eqref{exactseqA}. 
Thus we can consider $\pi$ as a representation of 
$G_{\pl_\circ}\times G^{\pl_\circ}$.
Since the group $G_{\pl_\circ}$ is Type I,
by \cite[Thm 1.8]{MackeyIndII}
we can decompose $\pi$ into a tensor product $\pi_{\pl_\circ}^{}\otimes \pi_{\pl_\circ}'$ of irreducible unitary representations of respective factors. We call $\pi_{\pl_\circ}$ the 
\emph{local component} of $\pi$ at  $\pl_\circ$.  

With slightly more work (see \cite[Sec. I.1.2]{MW}) one can show that indeed $G_\A$ is isomorphic to a quotient of a restricted product $\prod'_{\pl\in\Places}G_{\pl}$. We can then inflate a representation of $G_\A$ to one of 
$\prod'_{\pl\in\Places}G_{\pl}$, and use the restricted tensor product decomposition with respect to  the latter group.
\end{rmk}
\section{Functional calculus on nilpotent Lie groups}

\label{FuncCalcSec}

Throughout this section $N$ will be a simply connected nilpotent Lie group. 
 Fix a Haar measure $dn $ on $N$. For any $p\geq 1$, we denote the Banach space of complex-valued $p$-integrable functions on $N$ by $L^p(N)$. For every 
$f_1,f_2\in L^1(N)$,  set
$f_1*f_2(a):=\int_N f_1(n)f_2(n^{-1}a)dn$. 
The conjugate-linear involution $f\mapsto f^\dagger$ of $L^1(N)$  is defined by $f^\dagger(n):=\overline{f(n^{-1})}$ for every $n\in N$.
For every $f\in L^1(N)$, set 

\[
f^{*n}:=
\underbrace{f*\cdots *f}_{\text{$n$ times}}
\ \text{ and }\
e^{*f}:=\sum_{n=0}^\infty \frac{f^{*n}}{n!}
\in L^1(N).
\]
Since $N$ is simply connected, the exponential map 
is a diffeomorphism from $\g n:=\Lie(N)$ onto $N$. 
\begin{dfn}
\label{dfnSAN}
The \emph{Schwartz algebra} of $N$, denoted by $\sS(N)$, is the space of functions $f:N\to\C$ such that $f\circ \exp$ is a Schwartz function on $\g n$ in the sense of \cite[Sec. 25]{Treves}.
\end{dfn}
It is well known 
(for example, see \cite[Sec. 6.2]{FujiwaraLudwig})
that $\sS(N)$ is a subalgebra of the convolution algebra $L^1(N)$.

Given any $\phi\in C^\infty_c(\R)$, we set $\widehat{\phi}(t):=\int_{-\infty}^\infty e^{-ist}\phi(s)ds$ for every $t\in\R$.
Furthermore,  for every bounded self-adjoint operator $A$ on a Hilbert space, we define $\phi(A)$ by functional calculus, as in \cite[Chap. VII]{ReSi}.
\begin{prp}
\label{prp-funcal}
Let $f\in C_c(N)$, and let $\phi\in C_c^\infty(\R)$. Assume that
$f=f^\dagger$ and $\phi(0)=0$. Then the
 following statements hold.
\begin{itemize}
\item[\rm (i)]
The integral 
$
\phi\{f\}:=\frac{1}{2\pi }\int_{-\infty}^\infty
\widehat{\phi}(t) e^{*itf}dt
$
converges absolutely in $L^1(N)$, and $\phi\{f\}\in \sS(N)$.
\item[\rm (ii)] For every unitary representation $(\pi,\ccH)$ of $N$, we have
$\pi(\phi\{f\})=\phi(\pi(f))$.
\end{itemize}
\end{prp}
\begin{proof}
Absolute convergence of the integral in (i) follows from \cite[Lem. 7]{DixORF}. The fact that $f\in\sS(N)$ follows from \cite[Thm 2.6]{LRS}. (It was pointed out to us by Professor Jean Ludwig that the latter statement was first proved by A. Hulanicki.)
Part (ii) follows from \cite[Lem. 7]{DixORF}. 
\end{proof}

\section{The unitary dual of a unipotent $p$-adic group}
Let $\bfN$ be a unipotent algebraic group defined over $\Q$, and let $N$ be the group of $\Q_\pl$-points of $\bfN$ for some place $\pl\in\Places-\{\infty\}$.
As before, we denote the unitary dual of $N$ by $\widehat{N}$. Members of $\widehat{N}$ are  equivalence classes of irreducible unitary representations of $N$. We now recall the definition of the Fell topology of $\widehat{N}$. 
%Henceforth, we will denote the Fell topology by $\EuScript F$.
For any $(\pi,\ccH)\in\widehat{N}$, $k\in\N$, $v_1,\ldots,v_k,w_1,\ldots,w_k\in\ccH$, $\eps>0$, and $\Omega\subseteq N$ compact, 
we define 
\begin{equation}
\label{basopuni}
\cU(\pi,\Omega,\eps;v_1,\ldots,v_k;w_1,\ldots,w_k)
\end{equation}
to be the set of all $(\sigma,\ccK)\in\widehat{N}$ for which there exist $v_1',\ldots,v'_k,w'_1,\ldots,w'_k\in\ccK$ such that
\begin{equation}
\label{lagpilagsig}
\left|\lag \pi(g)v_i,w_j\rag -\lag\sigma(g)v'_i,w'_j\rag\right|<\eps\text{ for every }g\in \Omega\text{ and every }1\leq i,j\leq k.
\end{equation}
The sets defined in \eqref{basopuni} constitute a base  for the Fell topology. 

In \cite{GelfandKazhdan}, Gelfand and Kazhdan define a smooth version of the Fell topology on $\widehat{N}$.  This topology, which will will refer to by the \emph{Gelfand--Kazhdan topology}, is also the one that is used in \cite{bosa}. The main goal of this section is to prove that the Fell topology and the Gelfand--Kazhdan topology of $\widehat{N}$ are  indeed the same. 

In order to define the Gelfand--Kazhdan topology, we need the notion of the \emph{smooth dual} of $N$. 
Let $(\sigma,V)$ be a representation of $N$ on a complex vector space $V$. Recall that a  vector $v\in V$ is called \emph{smooth} if its stabilizer
contains a compact open subgroup of $N$.
Here and thereafter $V^K$ denotes the subspace of $K$-fixed vectors of in $V$.
The representation $(\sigma,V)$  is called \emph{smooth} if  every  $v\in V$ is smooth. We say $(\sigma,V)$ is \emph{admissible} if $\dim(V^K)<\infty$ for every compact open subgroup $K\subseteq N$. 
A smooth representation $(\sigma,V)$ is called \emph{pre-unitary} if it is equipped with an $N$-invariant positive definite Hermitian form. 
The set of equivalence classes of (algebraically) irreducible smooth representations of $N$ is called the \emph{smooth dual} of $N$.

Let $(\sigma,V)$ be an irreducible smooth representation of $N$. From \cite{VD} it follows that 
$(\sigma,V)$ is admissible and pre-unitary. Let 
$(\hat\sigma,\hat V)$ denote the unitary representation of $N$ corresponding to $(\sigma,V)$, so that
$\hat V$ is the Hilbert space completion of $V$.

\begin{prp}
\label{prp-asgnNN0}
The assignment $(\sigma,V)\mapsto (\hat\sigma,\hat V)$ results in a bijective correspondence between the smooth dual and the unitary dual of $N$.
\end{prp}
\begin{proof}
\textbf{Step 1.}
Let $(\sigma,V)$ be an irreducible smooth representation of $N$, and 
let $\hat V^\infty$ denote the subspace of smooth vectors of the unitary representation 
$(\hat\sigma,\hat V)$. 
First we prove that $\hat V^\infty=V$.
Clearly $V\subseteq\hat V^\infty$. To prove the reverse inclusion, let $v\in \hat V^\infty$ and choose a compact open subgroup $K\subseteq N$ such that $K$ lies in the stabilizer of $v$. We can write $V=V^K\oplus V(K)$, where $V(K)$ is the kernel of the projection \[
P_K:V\to V\ ,\ 
P_Kw:=\int_K\sigma(n)w dn.
\] 
Here $dn$ denotes the Haar measure of $N$ which satisfies $\int_Kdn=1$. Let $\lag\cdot,\cdot\rag$ denote the inner product of $\hat V$.
If $v\not\in V^K$ then after replacing $v$ by $v-v_0$, where $v_0$ is the orthogonal projection of $v$ on $V^K$, we can assume that $\lag v,V^K\rag=0$. By invariance of the inner product of $\hat V$,  we obtain 
$\lag v,V(K)\rag=0$. It follows that $\lag v,V\rag=0$, which is a contradiction since $V$ is dense in $\widehat{V}$.

\textbf{Step 2.} We show that $(\sigma,V)\mapsto (\hat\sigma,\hat V)$ is well-defined, that is, 
if $(\sigma,V)$ is an irreducible smooth representation, then  $(\hat\sigma,\hat V)$ is an irreducible unitary representation. Suppose, on the contrary, that $\hat V=\ccV_1\oplus \ccV_2$ where $\ccV_1$ and $\ccV_2$ are non-zero closed $N$-invariant subspaces of the Hilbert space $\hat V$. Then $V=\hat V^\infty=\ccV_1^\infty\oplus\ccV_2^\infty$. Since the subspace of smooth vectors of a unitary representation is dense, both
$\ccV_1^\infty$ and $\ccV_2^\infty$ are non-zero. It follows that $(\sigma,V)$ is reducible, which is a contradiction.

\textbf{Step 3.} We show that the assignment $(\sigma,V)\mapsto (\hat\sigma,\hat V)$  is surjective.
Fix an irreducible unitary representation $(\sigma,\ccV)$ of $N$.  It is now enough to show that 
the smooth representation of $N$ on $V:=\ccV^\infty$ (the space of smooth vectors of $\ccV$)
 is (algebraically) irreducible.
Assume that $W\subsetneq V$ is an $N$-invariant subspace. 
As in Step 1, for every compact open subgroup $K\subseteq N$ we can write $V=V^K\oplus V(K)$.
Note that $W=W^K\oplus (W\cap V(K))$. Since $W\neq V$, we can choose $K$ such that $W^K\subsetneq V^K$. 
By the remark made at the end of \cite[Sec. 5]{VD}, $V$ is admissible. Thus
$\dim (V^K)<\infty$, and therefore we can choose $w\in V^K$ such that $\lag w,W^K\rag=0$. As in Step 1, we have $\lag w,V(K)\rag=0$, so that $\lag w,W\rag=0$. It follows that the closure of $W$ in $\ccV$ is an 
$N$-invariant subspace which does not contain $w$. From irreducibility of $(\sigma,\ccV)$ it follows that $W=\{0\}$.
This completes the proof of irreducibility of $V$.

\textbf{Step 4.} We show that the assignment $(\sigma,V)\mapsto (\hat\sigma,\hat V)$ is an injection. To this end, we need to show that if $(\sigma_1,V_1)$ and $(\sigma_2,V_2)$ are algebraically equivalent smooth representations, then the unitary representations $(\hat\sigma_1,\hat V_1)$ and $(\hat\sigma_2,\hat V_2)$ are unitarily equivalent. Let $T:V_1\to V_2$ be a linear map such that $T\sigma_1(n)=\sigma_2(n)T$ for every $n\in N$. Then the inner product $\lag v,w\rag' :=\lag Tv,Tw\rag$ on $V_1$ is also $N$-invariant.  By 
\cite[Prop. 2.1.15]{casselman}, up to scalar there exists at most one invariant inner product on $V_1$. It follows that up to a scalar, $T$ is an isometry. Therefore $T$ extends to an intertwining operator $\hat{V}_1\to \hat{V}_2$ by continuity.
\end{proof}
We are now ready to give the definition of the Gelfand--Kazhdan topology of $\widehat{N}$.  Because of Proposition \ref{prp-asgnNN0}, it is enough to define this topology on the smooth dual of $N$.
The Gelfand--Kazhdan topology is generated by the base of open sets  \[
\cU:=\cU(\pi,\Omega,\eps;v_1,\ldots,v_k;\lambda_1,\ldots,\lambda_k),
\]   
where  $(\pi,V)$ is an irreducible smooth representation of $N$, $k\in \N$, $\eps>0$, $\Omega\subseteq N$ is compact,
$v_1,\ldots,v_k\in V$, and
$\lambda_1,\ldots,\lambda_k\in V^*$ (the algebraic dual of the vector space $V$). The definition of $\cU$ is similar to the definition of the base sets of the Fell topology, with relation \eqref{lagpilagsig}  replaced by 
\begin{equation}
\left|\lambda_i(\pi(g)v_j)-\lambda'_i(\pi(g)v_j')
\right|<\eps\text{ for every }g\in N\text{ and every }1\leq i,j\leq k.
\end{equation}

\begin{rmk}
\label{rmkmatcoeff} Let $(\sigma,V)$ be an irreducible smooth representation of $N$. Let $v_1,\ldots,v_k\in V$, and let $\lambda\in V^*$. Fix a compact subset $\Omega\subseteq N$. Since $N$ is a union of compact open subgroups, there exists a compact open subgroup 
$L\subseteq N$ such that $\Omega\subseteq L$. It follows that the vector space \[
W:=\spn_\C\left\{\pi(g)v_i\,:\,g\in \Omega\text{ and }1\leq i\leq k\right\}
\] is finite dimensional. Recall that $(\pi,V)$ is pre-unitary \cite{VD}. If $\lag\cdot,\cdot\rag$ denotes the inner product of $V$, then it follows immediately that there exists a vector $v\in V$ such that $\lambda(\pi(g)v_i)=\lag \pi(g)v_i,v\rag$ for every $g\in \Omega$ and every $1\leq i\leq k$. 
\end{rmk}

\begin{prp}
\label{prp-topol}
The Fell topology and the Gelfand--Kazhdan topology on $\widehat{N}$ are identical.
\end{prp}
\begin{proof}
The proof is straightforward and left to the reader. It  follows from Remark \ref{rmkmatcoeff} and the fact that in any unitary representation, an arbitrary matrix coefficient can be approximated uniformly by matrix coefficients of smooth vectors.  
\end{proof}
Recall 
that $U_\pl$ denotes the group 
defined in \eqref{ngammatowerlocal}.
\begin{cor}
\label{cor-uoverAdu}
For every $\pl\in\Places$, the unitary dual 
 $\widehat{U}_\pl$, equipped with the Fell topology, is homeomorphic to 
the quotient space $\g u_\pl^*/\Ad^*(U_\pl)$. 
\end{cor}
\begin{proof}
For $\pl=\infty$, this is proved in \cite{brown}.
For $\pl\in\Places-\{\infty\}$, it is shown in \cite[Thm 3.1]{bosa} 
that
$\g u_\pl^*/\Ad^*(U_\pl)$ is homeomorphic to 
$\widehat{U}_\pl$ equipped with the Gelfand--Kazhdan topology. The corollary now follows from Proposition \ref{prp-topol}.
\end{proof}

\section{Existence of separating Schwartz functions}
\label{sec-existe}
For every $\pl\in\Places$, recall that $U_\pl$ is isomorphic to  the group of $\Q_\pl$-points of the unipotent algebraic group $\bfU$ defined in Section \ref{therankparabolic}.
From \cite[Chap. IV]{DG} it follows that $U_\infty$ is a simply connected nilpotent Lie group. 
% See also Chap XV of Milne's affine group schemes.
As 
in Definition \ref{dfnSAN}, 
we denote the algebra of Schwartz
functions on $U_\infty$ by $\sS(U_\infty)$. 
For $\pl\in\Places-\{\infty\}$, we set $\sS(U_\pl):=C_c^\infty(U_\pl)$, where
$C_c^\infty(U_\pl)$ is the convolution algebra of 
compactly supported locally constant complex-valued functions on $U_\pl$. 

Given $\pl\in \Places$ and a unitary representation 
$\pi$ of $U_\pl$, we  set
\begin{equation}
\label{eqGAL}
\pi(\psi):=
\int_{U_\pl}\psi(n)\pi(n)dn\text{ for every }\psi\in\sS(U_\pl),
\end{equation}
where $dn$ is a Haar measure  on $U_\pl$.
\begin{prp}
\label{vinfSU}
Fix $\pl\in \Places$. Let $S\subseteq \widehat{U}_\pl$ be a closed subset with respect to the Fell topology, and let $\pi\in \widehat{U}_\pl- S$. Then there exists an element $\psi\in \sS(U_\pl)$ such that
$\pi(\psi)\neq 0$, but $\sigma(\psi)=0$ for every $\sigma\in S$.
\end{prp}
\begin{proof}
We consider two separate cases.

\textbf{Case 1:} $\pl=\infty$. 
 Let $C^*(U_\infty)$ denote the $C^*$-algebra of $U_\infty$. Since $C^*(U_\infty)$ is CCR
(see \cite[Thm 11]{Moore} and the references therein), the Fell topology on $\widehat{U}_\infty$ 
 and the hull-kernel topology on the dual of $C^*(U_\infty)$ are identical \cite[Sec. 7.2]{Folland}. It follows that there exists an element $a\in C^*(U_\infty)$ such that $\|\pi(a)\|=1$, whereas $\sigma(a)=0$ for every $\sigma\in S$. Substituting $a$ by $aa^*$ if necessary, we can assume that $a=a^*$. 
 
 Let $C^\infty_c(U_\infty)$ denote the space of functions $U_\infty\to\C$ which are smooth and have compact support. Since $C^\infty_c(U_\infty)$ 
is a dense subspace of $C^*(U_\infty)$, we can choose
$f\in C^\infty_c(U_\infty)$ 
satisfying $f=f^\dagger$
such that $\|\pi(f)\|=1$, whereas 
$\|\sigma(f)\|<\frac{1}{4}$ for every $\sigma\in S$.  Next we choose $\phi\in C^\infty_c(\R)$ such that  
\[
\supp(\phi)\subseteq[-2,2],\quad
\phi\big|_{[-\frac{1}{4},\frac{1}{4}]}=0,\quad
\text{and}\quad 
\phi\big|_{[-\frac{5}{4},-\frac{3}{4}]}=
\phi\big|_{[\frac{3}{4},\frac{5}{4}]}=
1.
\]
Now set $\psi:=\phi\{f\}$. By Proposition
\ref{prp-funcal}(i), we have $\psi\in\sS(U_\infty)$. 
Set $A:=\pi(f)$ and note that $A$ is self-adjoint because $f=f^\dagger$. It follows that the spectral radius of $A$ is equal to $\|A\|=1$. Since
$\phi(\pm 1)=1$, we obtain $\|\phi(A)\|=\sup\{|\phi(x)|\,:\,x\in\mathrm{Spec}(A)\}>0$.
Consequently,  
Proposition \ref{prp-funcal}(ii) implies that $\pi(\psi)=\phi(A)\neq 0$. Similarly, for every $\sigma \in S$, the spectral radius of $\sigma(f)$ is equal to $\|\sigma(f)\|<\frac{1}{4}$. It follows that  $\phi$ vanishes on the spectrum of $\sigma(f)$, and therefore
$\sigma(\psi)=\phi(\sigma(f))=0$.

\textbf{Case 2:} $\pl\in\Places-\{\infty\}$.
By Proposition \ref{prp-asgnNN0} and Proposition \ref{prp-topol}, we can assume that $(\pi,V)$ is an irreducible smooth representation of $N$. The strategy of the proof is to use  the results of \cite{GelfandKazhdan}. 
Choose a compact open subgroup 
$K_0\subseteq U_\pl$ such that 
$V^{K_0}\neq \{0\}$, 
and fix a sequence \[
K_0\subseteq K_1\subseteq \cdots \subseteq K_n\subseteq\cdots
\] of compact open subgroups 
of $U_\pl$
such that 
$U_\pl=\bigcup_{n=0}^\infty K_n$.
Let $\cH_{K_0}$ denote the convolution algebra of $K_0$-bi-invariant compactly supported complex-valued functions on $U_\pl$, and let $\cH_{K_0}^{K_n}$ denote the subalgebra of $\cH_{K_0}$ which consists of functions whose support lies in $K_n$.
By
\cite[Prop. 4]{GelfandKazhdan},
the algebra $\cH_{K_0}^{K_n}$ is isomorphic to the commutant of the image of the group
algebra of $K_n$ in the induced representation $\mathrm{Ind}_{K_0}^{K_n}1$, and therefore it is a finite dimensional semisimple associative algebra.
Since $\cH_{K_0}\subseteq\sS(U_\pl)$,
 for every smooth representation $\sigma$ of $U_\pl$ and every $\phi\in\cH_{K_0}$ we define $\sigma(\phi)$ as in \eqref{eqGAL}.
%For every $\phi\in\cH_{K_0}$ and every smooth representation $(\sigma,W)$ of $U_\pl$, we define \[\sigma(\phi)v:=\frac{1}{\mu(K_0)}\int_{K_0}f(g)\pi(g)v\,d\mu(g),\] 
%where $\mu$ denotes the Haar measure of $U_\pl$.   

According to \cite[Thm 4]{Moore}, the group $U_\pl$ is CCR, and therefore every point in $\widehat{U}_\pl$ is closed \cite[Sec. 7.2]{Folland}.  Now let $\widehat{U}_\pl^{K_0}$ denote the subset of $\widehat{U}_\pl$ consisting of irreducible representations $(\sigma,W)$ such that $W^{K_0}\neq \{0\}$. We equip 
$\widehat{U}_\pl^{K_0}$ with the topology induced by
 the Fell topology of $\widehat{U}_\pl$. 
Then $S\cap \widehat{U}_\pl^{K_0}$ and $\{\pi\}$ are closed subsets of $\widehat{U}_\pl^{K_0}$, and therefore by 
 \cite[Thm 6]{GelfandKazhdan} and
\cite[Prop. 18]{GelfandKazhdan} there exists an $n\in \N$ such that for every $(\sigma,W)\in S\cap \widehat{U}_\pl^{K_0}$, the 
$\cH_{K_0}^{K_n}$-modules $W^{K_0}$ and $V^{K_0}$ are disjoint.  By Artin--Wedderburn theory, 
\begin{equation}
\label{HKniso}
\cH_{K_0}^{K_n}\cong \mathrm M_{d_1\times d_1}(\C)\times\cdots\times \mathrm M_{d_m\times d_m}(\C)
\end{equation}
for some integers $d_1,\ldots,d_m\geq 1$, where 
$\mathrm M_{d\times d}(\C)$ is the associative algebra of $d\times d$ matrices with complex entries. The irreducible modules of $\cH_{K_0}^{K_n}$
are the standard modules $\C^{d_i}$, $1\leq i\leq m$, of the ideals $\mathrm M_{d_i\times d_i}(\C)$. From 
\eqref{HKniso} and 
disjointness of $V^{K_0}$ and $W^{K_0}$ it follows that there exists an 
idempotent $\psi\in \cH_{K_0}^{K_n}$ such that $\pi(\psi) V^{K_0}\neq \{0\}$, whereas $\sigma(\psi)W^{K_0}=\{0\}$ for every $(\sigma,W)\in S\cap \widehat{U}_\pl^{K_0}$. 
It follows that $\pi(\psi)\neq 0$, whereas \[
\sigma(\psi)W=\sigma(\psi * \psi)W=\sigma(\psi)^2W\subseteq\sigma(\psi)W^{K_0}=\{0\}.
\]
But also when $W^{K_0}=\{0\}$, we have 
$\sigma(\phi)W\in W^{K_0}=\{0\}$
for every $\phi\in \cH_{K_0}$, and in particular
$\sigma(\psi)=0$.
\end{proof}

\section{Kirillov theory for $U_\A$}
\label{KirillovA}

As shown in \cite{Moore}, the group $U_\A$ is \emph{not} of Type I because it is non-abelian. However, 
it is shown in 
 \cite[Thm 11]{Moore} 
that the decomposition of the representation of
$U_\A$ by right translation on $L^2(U_\Q\bls U_\A)$ can be described by means of Kirillov theory. We now recall  Moore's result from  \cite[Thm 11]{Moore}. Let $\g u_\Q$ denote the Lie algebra of $U_\Q$, and let $\g u_\Q^*$ denote the dual of $\g u_\Q$. 
Fix a place $\pl\in\Places$. Every $\mu\in \g u_\Q^*$ can be extended in a unique way to a linear functional $\mu_\pl\in\g u_\pl^*:=\Lie(U_\pl)$. Let $\rho_{\mu_\pl}$ denote the irreducible unitary representation of $U_\pl$ that corresponds to the coadjoint orbit associated to $\mu_\pl$.
Now let $\left(\sfR',L^2(U_\Q\bls U_\A)\right)$ denote the representation of $U_\A$ on $L^2(U_\Q\bls U_\A)$ by right translation.
It is shown in \cite[Thm 11]{Moore} that  
$\left(\sfR',L^2(U_\Q\bls U_\A)\right)$ decomposes as a multiplicity-free direct sum of unitary representations 
\begin{equation}
\label{rhomRTP}
\rho_\mu:=\otimes_{\pl\in\Places}^{}
\rho_{\mu_\nu}\text{ for all }\mu\in\g u_\Q^*.
\end{equation}
In the next section we will need the following lemma.

\begin{lem}
\label{orbitdim}
Let $\mu\in \g u_\Q^*$, and
for every place $\pl\in\Places$
 let $\mathcal O_{\mu_\pl}\subseteq \g u_\pl^*$ 
denote the $U_\pl$-orbit of $\mu_\pl$, where 
$\mu_\pl\in\g u_\pl^*$ is the canonical extension of 
$\mu$. Then 
$\mathcal O_{\mu_\pl}^{}$
is an analytic $\Q_\pl$-submanifold of $\g u_\pl$, and 
$\dim(\mathcal O_{\mu_\pl}^{})$
is independent of the place 
$\pl$. \end{lem}
\begin{proof}
The action of $\bfU$ on $\g u^*$ is algebraic and defined over $\Q$. Since $\mu\in\g u_\Q^*$, the stabilizer of $\mu$ is an algebraic group $\bfS\subseteq\bfU$ that is defined over $\Q$.
Fix $\pl\in\Places$. From 
the proof of Lemma \ref{dimcoadj}
 it follows that 
$\mathcal O_{\mu_\pl}^{}$ is an analytic 
$\Q_\pl$-submanifold of $\g u_\pl^*$. Let 
$S_\pl$ be the stabilizer of $\mu_\pl$ in $U_\pl$. 
Thus $S_\pl$  is the set of $\Q_\pl$-points of $\bfS$, 
and hence it is an analytic $\Q_\pl$-manifold  of dimension 
$\dim(\bfS)$ (see \cite[Sec. 3.1]{PlatonovRapinchuk}). On the other hand,
by \cite[Sec. II.4.5]{Serre} we have 
$\dim(S_\pl)+\dim(\mathcal O_{\mu_\pl})=
\dim(U_\pl)=\dim(\bfU)$. Consequently, 
$\dim(\mathcal O_{\mu_\pl})=\dim(\bfU)-\dim(\bfS)$ is independent of $\pl\in\Places$.
\end{proof}

\section{Rank for global representations}
Recall the definition of $r:=r(\bfG)$ from 
Section \ref{therankparabolic}.
%For every place $\pl$, let $\sS(U_{\pl})$ be the Schwartz space of $U_{\pl}$. When $\pl<\infty$, this is the space of locally constant functions with compact support, whereas when $\pl=\infty$, it is the space of maps $\phi:U_\infty\to\C$ such that $\phi\circ\exp$ is a Schwartz function on the vector space $\g u_\infty:=\Lie(U_\infty)$, where $\exp$ denotes the exponential map. (Note that  $\exp:\g u_\infty\to U_\infty$ is a diffeomorphism since $U_\infty$ is simply connected, as it is the group of $\R$-points of a unipotent algebraic group.) For every $\phi\in \sS(U_{\pl})$ and every unitary representation $(\sigma,\ccK)$ of $U_{\pl}$, we set \[\sigma(\phi)v:=\int_{U_{\pl}}\phi(g)\sigma(g)vdg\ \text{ for every }v\in\ccK.\] 
For every $\pl\in\Places$ and $0\leq d\leq r(\bfG)$, let
$\widehat{U}_{\pl}[d]\subseteq\widehat{U}_{\pl}$ denote the set consisting  of irreducible 
unitary representations that correspond to coadjoint orbits of dimension at most $\dim(U_{1,\pl})+\cdots+\dim(U_{d,\pl})-d$. 
Note that $\widehat{U}_\pl(d)\subseteq \widehat{U}_\pl[d]$, where $\widehat{U}_\pl(d)$
is defined in Definition \ref{DfnU(d)}.
\begin{lem}
$\widehat{U}_{\pl}[d]$ is a closed subset of $\widehat{U}_{\pl}$ for every $0\leq d\leq r(\bfG)$ and every $\pl\in\Places$.
\end{lem}
\begin{proof}
Let $U_\pl(\lambda)$ denote the stabilizer of   $\lambda\in\g u_\pl^*$ in $U_\pl$.
By Corollary \ref{cor-uoverAdu}, $\widehat{U}_\pl$ is homeomorphic to 
 $\g u_\pl^*/\Ad^*(U_\pl)$.
Therefore it suffices to prove that for every $n\in\N$,  the set
\[
T_n:=\left\{\lambda\in\g u_\pl^* \ :\ 
\dim \left(U_\pl(\lambda)\right)<n
\right\}
\] is an open subset of $\g u_\pl^*$. 
For every $\lambda\in \g u_\pl$,  
let $h_\lambda:U_\pl\to \g u_\pl^*$ be defined by $h_\lambda(g):=\Ad^*(g)\lambda$. Then $\dim(U_\pl(\lambda))=\dim\g u_\pl-\mathrm{rank}(\dd h_\lambda(\yek))$, where $\dd h_\lambda$ is the differential of $h_\lambda$. Since the map $\lambda\mapsto\mathrm{rank}(\dd h_\lambda(\yek))$ is a lower semi-continuous function of  $\lambda$, 
 the complement of $T_n$ is open.\end{proof}
For every $\pl\in\Places$, set
\[
J_{d,\pl}:=\left\{
\phi\in\sS(U_{\pl})\ :\ \sigma(\phi)=0\text{ for every }
\sigma\in\widehat{U}_{\pl}[d]\right\}
.\]
%\begin{rmk}
%\label{rmk-pinv}
%Recall that $P_\pl$ denotes the set of $\Q_\pl$-points of the parabolic subgroup $\bfP\subseteq \bfG$. 
%Every $J_{\pl,d}$ is $P_\pl$-invariant, that is,
%for every $\psi\in J_{d,\pl}$ and $p\in P_\pl$ we have $\psi^p\in J_{d,\pl}$, where $\psi^p(n):=\psi(pnp^{-1})$. This is because the coadjoint action of  $p\in P_\pl$ on $\g u_\pl^*$ does not change the dimension of a coadjoint orbit, so that if $\sigma\in\widehat{U}_\pl[d]$ then $\sigma^p\in \widehat{U}_\pl[d]$, where $\sigma^p(n):=\sigma(pnp^{-1})$.
%\end{rmk}
\begin{lem}
\label{lemrkdsii}
Fix $\pl\in\Places$. Let $(\sigma,\ccH)$ be  a unitary representation of $U_\pl$, and let \[
\sigma=\int^\oplus_{\widehat{U}_\pl}n_\tau\tau d\mu(\tau)
\] be the direct integral decomposition of $\sigma$. For $1\leq d\leq r(\bfG)$, 
the following statements are equivalent.
\begin{itemize}
\item[\rm(i)] $\supp(\mu)\subseteq \widehat{U}_\pl[d]$.
\item[\rm(ii)] $\sigma(\phi)=0$ for every 
$\phi\in J_{d,\pl}$.
\end{itemize}
\end{lem}
\begin{proof}
(i)$\RA$(ii): 
%Since $\widehat{U}_\pl(d')\subseteq \widehat{U}_\pl[d]$ for every $d'\leq d$, we can decompose $\sigma$ as a direct integral $ \sigma:=\int^\oplus_{\widehat{U}_\pl[d]} n_\tau\tau d\tau$, where $n_\tau\in\N\cup\{\infty\}$.
From  \cite[Sec. 14.9.2]{WallachII} and the definition of $J_{\pl,d}$ it follows that
\[
\sigma(\phi)=
\int^\oplus_{\widehat{U}_\pl[d]}
n_\tau\tau(\phi) d\mu(\tau)=0.
\]
(ii)$\RA$(i): We prove the contrapositive, that is, if (i) is false then (ii) is false. Suppose that 
%$\mathrm{rank}(\pi_\pl)>d$. If we decompose$\sigma$ as a direct integral $\sigma:=\int^\oplus_{\widehat{U}_\pl}n_\tau\tau d\tau $, then by Theorem  \ref{mainsalduke} 
the support of $\mu$  does not lie inside $\widehat{U}_\pl[d]$. Since $\widehat{U}_\pl[d]$ is a closed subset of $\widehat{U}_\pl$, we obtain $\mu\big(\widehat{U}_\pl-\widehat{U}_\pl[d]\big)> 0$. 
It follows that there exists some $(\tau_{\circ},\ccH_{\tau_\circ}) \in\widehat{U}_\pl-\widehat{U}_\pl[d]$ such that $\mu(\cU)> 0$ for every open neighborhood $\cU$ of $(\tau_{\circ},\ccH_{\tau_\circ}) \in\widehat{U}_\pl-\widehat{U}_\pl[d]$  (because otherwise, 
since $\widehat{U}_\pl$ is second countable \cite[Prop. 3.3.4]{Dixmierbook},
we will find a covering of 
$\widehat{U}_\pl-\widehat{U}_\pl[d]$ by countably many null open sets). By Proposition \ref{vinfSU}, there exists an element $\psi\in J_{d,\pl}$ such that $\tau_\circ(\psi)\neq 0$.
Set  $\psi^\dagger(n):=\overline{\psi(n^{-1})}$ for $n\in U_\pl$.
Without loss of generality we can assume that $\psi=\psi^\dagger$, because otherwise we can replace $\psi$ by either $i(\psi-\psi^\dagger)$ or $\psi+\psi^\dagger$.
After scaling $\psi$ by a real number, we can also assume that $\|\tau_\circ(\psi)\|=1$.
Since $\tau_\circ(\psi)$ is self-adjoint, 
we can choose
$v\in \ccH_{\tau_\circ}$ such that
 $\|v\|=1$ and $\left|\lag\tau_\circ(\psi)v,v\rag\right|>\frac{3}{4}$.

Fix $\eps>0$ such that 
$\eps\left(2+\eps+\|\psi\|_{L^1(U_\pl)}\right)<\frac{3}{4}$, and choose a compact subset $\Omega\subseteq U_\pl$ such that $\yek\in \Omega$ and 
$\|\psi-\chi_\Omega^{}\psi\|_{L^1(U_\pl)}<\eps$,  where $\yek$ denotes the neutral element of $U_\pl$ and $\chi_\Omega^{}$ denotes the 
characteristic function of $\Omega$. Set $\cU:=\cU(\tau_\circ,\Omega,\eps;v;v)$, defined as in \eqref{basopuni}. For every  $(\tau,\ccH_\tau)\in \cU$, there exists a vector $w\in \ccH_\tau$ such that 
 \[\sup\left\{|\lag\tau_\circ(n)v,v\rag-\lag\tau(n)w,w\rag |\,:\,n\in\Omega\right\}<\eps.
\]
In particular, setting $g=\yek$ we obtain $\|w\|^2<1+\eps$.
Next set
\[a:=|\lag\tau_\circ(\psi)v,v\rag -\lag\tau_\circ(\chi_\Omega^{}\psi)v,v\rag|
\text{ and }b:=|\lag\tau(\psi)w,w\rag -\lag\tau(\chi_\Omega^{}\psi)w,w\rag|
.\]
By the choice of $\Omega$, we have 
$a\leq \|\psi-\chi_\Omega^{}\psi\|_{L^1(U_\pl)}< \eps$ and $b\leq (1+\eps)\|\psi-\chi_\Omega^{}\psi\|_{L^1(U_\pl)}< \eps(1+\eps)$.
Now set 
$c:=|\lag\tau_\circ(\chi_\Omega\psi)v,v\rag-
\lag\tau(\chi_\Omega\psi)w,w\rag|$.
Then
\[
c\leq \int_{\Omega}|\psi(n)|\cdot
|\lag\tau_\circ(n)v,v\rag-\lag \tau(n)w,w\rag|dn
< \eps\|\psi\|_{L^1(U_\pl)}
\]
By the above estimates and the triangle inequality we obtain
$|\lag\tau_\circ(\psi)v,v\rag-\lag\tau(\psi)w,w\rag|\leq a+b+c<\frac{3}{4}$.
Since $|\lag \tau_\circ(\psi)v,v\rag|>\frac{3}{4}$, we obtain 
$\lag\tau(\psi)w,w\rag\neq 0$, and in particular $\tau(\psi)\neq 0$. Finally, since  
$\sigma(\psi)=
\int_{\widehat{U}_\nu}^\oplus
n_\tau\tau(\psi) d\mu(\tau)$, we obtain $\sigma(\psi)\neq 0$.
\end{proof}

\begin{lem}
\label{lem:contPhi1Dc}
Fix a place $\pl\in\Places$. Let 
$f\in L^2(G_\Q\bls G_\A)$ be of moderate growth, and let 
$\psi\in\sS(U_\pl)$. 
Then the map
\[
G_\Q\bls G_\A\to\C\ ,\ x\mapsto\int_{U_\nu}\psi(n)f(xn)dn
\]
is continuous and in $L^2(G_\Q\bls G_\A)$.
\end{lem}
\begin{proof}
Set $\Phi_1(x):=\int_{U_\nu}\psi(n)f(xn)dn$ and $\Phi_2(x):=\int_{U_\nu}
|\psi(n)f(xn)|dn$ for every $x\in G_\Q\bls G_\A$. 
Recall that by
Definition \ref{dfnAuFo}, a moderate growth element of $L^2(G_\Q\bls G_\A)$ is  a smooth map $G_\Q\bls G_\A\to \C$.

\textbf{Step 1.} 
We show that $\Phi_2(x)<\infty$ for every $x\in G_\Q\bls G_\A$.  If $\pl\in\Places-\{\infty\}$, then the statement is obvious since $\psi$ is compactly supported. Next assume that  $\pl=\infty$, 
and fix $x\in G_\Q\bls G_\A$. 
Let $\|\cdot\|^{}_{\g u_\infty}$ be a norm on $\g u_\infty$. By \eqref{mfboundeq} we can assume that
there exist $c_1,m_1>0$
such that 
\[
|f(x\exp(y))|\leq c_1(\|y\|^{}_{\g u_\infty}+1)^{m_1}
\text{ for every }y\in\g u_\infty.
\]
Since $\psi\in\sS(U_\infty)$, there exists a constant
$c_2>0$ such that \[
\psi(\exp(y))\leq c_2(\|y\|^{}_{\g u_\infty}+1)^{-m_1-2\dim(\g u_\infty)}\text{ for every }y\in\g u_\infty.
\]
 Since the Haar measure on $U_\pl$ is the pushforward of the Lebesgue measure of $\g u_\pl:=\Lie(U_\pl)$ via the exponential map, we obtain $\Phi_2(x)<\infty$.

\textbf{Step 2.} From Step 1 it follows that the integral defining $\Phi_1(x)$ is convergent for every $x\in G_\Q\bls G_\A$. In this step  we assume that 
$\pl\in\Places-\{\infty\}$, and we prove that $\Phi_1$ is a continuous map. (The case $\nu=\infty$ will be addressed in Step 3 below.) 
%Thus in this step we assume that $\pl\in\Places-\{\infty\}$.
Fix $x\in G_\Q\bls G_\A$. Our goal is to prove continuity of $\Phi_1$ at $x$. 
Set $\Omega:=\supp(\psi)$. For every $y\in G_\Q\bls G_\A$,
\begin{equation}
\label{eqPhi1Phi1y}
\Phi_1(x)-\Phi_1(y)=\int_\Omega 
\psi(n)(f(xn)-f(yn))dn.
\end{equation}
Fix any $\eps>0$. By continuity of the map
$\Omega\times (G_\Q\bls G_\A)\to \C$ defined as $(n,x)\mapsto f(xn)$, and by compactness of $\Omega$, there exists an open neighborhood $\cU\subseteq G_\Q\bls G_\A$ of $x$
such that 
\begin{equation}
\label{eqfxn-fyn|}
|f(xn)-f(yn)|<\frac{\eps}{\|\psi\|_{L^1(U_\pl)}}\ \text{ for every }y\in\cU\text{ and every }n\in\Omega.
\end{equation} 
From \eqref{eqfxn-fyn|}
and \eqref{eqPhi1Phi1y} it follows that 
$|\Phi_1(x)-\Phi_1(y)|<\eps$ for every $y\in\cU$.
Since $\eps>0$ is arbitrary, the latter inequality proves continuity of $\Phi_1$ at an arbitrary point
$x\in G_\Q\bls G_\A$.

\textbf{Step 3.} In this step we assume that  $\nu=\infty$, and we prove that  the map $\Phi_1$ is continuous.
Fix $x\in G_\Q\bls G_\A$. Our goal is to prove continuity of $\Phi_1$ at $x$. 
Suppose that $\Omega\subseteq U_\infty$ is an arbitrary   relatively compact open set. Then for every $y\in G_\Q\bls G_\A$, 
\begin{equation}
\label{v=infPHi1}
\Phi_1(x)-\Phi_1(y)=
\int_\Omega \psi(n)(f(xn)-f(yn))dn
+\int_{U_\infty-\Omega}\psi(n)(f(xn)-f(yn))dn.
\end{equation}
Choose any $\eps>0$. 
Our goal is to obtain upper estimates for the above integrals on $\Omega$ and $U_\infty- \Omega$ for a suitably chosen  $\Omega$.

Every
$y\in G_\Q\bls G_\A$ can be written as
 $y=xz_\mathrm{fin}z_\infty$ where 
$z_\mathrm{fin}\in G_{\A_\mathrm{fin}}$ and $z_\infty\in G_\infty$. If $y$ is chosen sufficiently close to $x$, then $z_\mathrm{fin}$ and $z_\infty$ will be sufficiently close to the neutral elements of 
$G_{\A_\mathrm{fin}}$ and $G_{\R}$, respectively.
Therefore smoothness of $f:G_\Q\bls G_\A\to \C$ and the growth bound \eqref{mfboundeq}   imply that there exists a sufficiently small neighborhood $\cU_1\subseteq G_\Q\bls G_\A$ of $x$ such that for every $y\in\cU_1$ and every $n\in U_\infty$,
\begin{align}
\label{Calfxn-fyn}
\notag
|f(xn)-f(yn)|&\leq |f(xn)|+|f(xz_\mathrm{fin}z_\infty n)|
=
|f(xn)|+|f(xz_\infty nz_\mathrm{fin})|
\\
&=
|f(xn)|+|f(xz_\infty n)|\leq 
c_{x,f}\|n\|^{m_f}+c_{x,f}\|z_\infty n\|
^{m_f}\leq c_3\|n\|^{m_f},
\end{align}
where $c_3>0$ is a constant. Fix a norm $\|\cdot\|^{}_{\g u_\infty}$ on $\g u_\infty$. 
From \eqref{Calfxn-fyn}
it follows that there exist $c_4,m_4>0$ such that
\begin{equation}
\label{xexpuus}
|f(x\exp(u))-f(y\exp(u))|\leq c_4
(\|u\|^{}_{\g u_\infty}+1)^{m_4}
\text{ for every }u\in \g u_\infty
\text{ and every }
y\in\cU_1.
\end{equation}
Since $\psi\in\sS(U_\pl)$, we can choose an $\Omega$ suitably large such that 
\begin{equation}
\label{xexpuus1}
|\psi(\exp(u))|\leq \frac{\eps}{c_4}(\|u\|^{}_{\g u_\infty}+1)^{-m_4-2\dim (\g u_\infty)}\text{ for every }u\in \g u_\infty\text{ such that }\exp(u)\in U_\infty-\Omega.
\end{equation}
Set 
$c_5:=\int_{U_\pl}(\|u\|_{\g u_\infty}+1)^{-2\dim(\g u_\infty)}du$. 
Since the Haar measure of $U_\infty$ is the pushforward of the Lebesgue measure of $\g u_\infty$, the latter integral is convergent. From \eqref{xexpuus} and \eqref{xexpuus1} it follows that   
\begin{equation}
\label{c5EQ}
\left|\int_{U_\infty-\Omega} 
\psi(n)(f(xn)-f(yn))dn\right|<c_5\eps
\,\text{ for every }y\in \cU_1.
\end{equation}

With an argument similar to the case $\pl\neq \infty$ in Step 2, 
we can show that there exists a
sufficiently small neighborhood $\cU_2\subseteq G_\Q\bls G_\A$ of $x$ (which depends on $\Omega$) such that for every $y\in \cU_2$ we have 
\begin{equation}
\label{c55EQ}
\left|\int_\Omega \psi(n)(f(xn)-f(yn))dn\right|<\eps.
\end{equation}
From \eqref{c5EQ} and \eqref{c55EQ} it follows that 
for every $y\in\cU_1\cap\cU_2$ we have
\[
\left|\int_{U_\infty} \psi(n)(f(xn)-f(yn))dn\right|<(c_5+1)\eps.
\]
The latter inequality implies continuity of $\Phi_1$ at $x$.

\textbf{Step 4.} We prove that $\Phi_1\in L^2(G_\Q\bls G_\A)$.
It is enough  to prove that 
$\Phi_2\in L^2(G_\Q\bls G_\A)$. 
Measurability of $\Phi_2$ follows from Fubini's Theorem. 
Next fix any $h\in L^2(G_\Q\bls G_\A)$. By Fubini's Theorem and the Cauchy--Schwarz inequality,
\begin{align*}
\int_{G_\Q\bls G_\A}|\Phi_2(x) h(x)|dx
&=
\int_{U_\pl}\left(\int_{G_\Q\bls G_\A}|\psi(n)f(xn)h(x)|dx\right)dn\\
&
\leq\int_{U_\pl}|\psi(n)|\cdot
\|f\|^{}_{L^2(G_\Q\bls G_\A)}
\cdot 
\|h\|^{}_{L^2(G_\Q\bls G_\A)}dn\\
&\leq
\|\psi\|_{L^1(U_\pl)}
\cdot
\|f\|^{}_{L^2(G_\Q\bls G_\A)}
\cdot 
\|h\|^{}_{L^2(G_\Q\bls G_\A)}dn.
\end{align*}
Thus the map 
$
h\mapsto \int_{L^2(G_\Q\bls G_\A)}\Phi_2(x)h(x)dx
$
is a bounded linear functional on $L^2(G_\Q\bls G_\A)$, hence by the Riesz representation theorem we obtain $\Phi_2\in L^2(G_\Q\bls G_\A)$.\end{proof}

Lemma \ref{lemfbn} below probably follows from standard results in the literature. 
We include a complete proof because we did not find a suitable reference. The tricky point is to use Fubini's Theorem carefully to justify that one 
can change the order of certain integrals.

Before stating  Lemma \ref{lemfbn}, we remind the reader that  by 
Definition \ref{dfnAuFo}, a moderate growth element  $f\in L^2(G_\Q\bls G_\A)$ is assumed to be a smooth map $f:G_\Q\bls G_\A\to \C$. In particular, the restriction 
$f\big|_{U_\Q\bls U_\A}$ is well-defined and continuous. 
Recall that $\sS(U_\pl)$ denotes the Schwartz space of $U_\pl$.
\begin{lem}
\label{lemfbn}
Fix a place $\pl\in\Places$. Let 
$f\in L^2(G_\Q\bls G_\A)$ be  of moderate growth.
Set $\sfR_{U_\pl}:=\sfR\big|_{U_\pl}$, where $\sfR$ denotes the representation of $G_\A$ on $L^2(G_\Q\bls G_\A)$ by right translation,  and let $\psi\in\sS(U_\pl)$. Then 
\begin{itemize}
\item[\rm (i)]
$
\left(
\sfR_{U_\pl}(\psi)f\right)(x)=
\int_{U_\nu}\psi(n)f(xn)dn
$
for almost every 
$
x\in G_\Q\bls G_\A.
$
\item[\rm (ii)] Let
$\sfR'$ denote the representation of $U_\A$ on $L^2(U_\Q\bls U_\A)$ by right translation, and set $\sfR'_{U_\pl}:=\sfR'\big|_{U_\pl}$. Then 
$
\left(
\sfR'_{U_\pl}(\psi)f\big|_{U_\Q\bls U_\A}\right)(x)=
\int_{U_\nu}\psi(n)f(xn)dn
$
for almost every 
$
x\in U_\Q\bls U_\A.
$

\end{itemize} 
\end{lem}
\begin{proof}
We will only give the proof of (i). The proof of (ii) is similar and indeed somewhat easier, since $U_\Q\bls U_\A$ is compact.
Set $\Phi_1(x):=\int_{U_\nu}\psi(n)f(xn)dn$ %and $\Phi_2(x):=\int_{U_\nu}|\psi(n)f(xn)|dn$ 
for every $x\in G_\Q\bls G_\A$. 

Fix any
$h\in L^2(G_\Q\bls G_\A)$. 
By Lemma \ref{lem:contPhi1Dc}
 we are allowed to use Fubini's Theorem to write
\begin{align*}
\lag\sfR_{U_\pl}(\psi)f,h\rag
&=
\int_{U_\pl}\psi(n)
\lag\sfR_{U_\pl}(n)f,h\rag
dn=
\int_{U_\pl}
\int_{G_\Q\bls G_\A}
\psi(n)
f(xn)\oline{h(x)}dxdn\\
&=
\int_{G_\Q\bls G_\A}
\int_{U_\pl}
\psi(n)
f(xn)\oline{h(x)}dndx=
\int_{G_\Q\bls G_\A}\Phi_1(x)\overline{h(x)}dx=
\lag \Phi_1,h\rag.
\end{align*}
Since
$h\in L^2(G_\Q\bls G_\A)$ is arbitrary,
from the above calculation it follows that 
$\sfR_{U_\pl}(\psi)f=\Phi_1$ as elements of $L^2(G_\Q\bls G_\A)$.
\end{proof}

Next we state and prove our main theorem (from the introduction). Recall that 
$(\sfR,L^2(G_\Q\bls G_\A))$ denotes the unitary representation of $G_\A$ on $L^2(G_\Q\bls G_\A)$ by right translation. Furthermore, recall the
definition of $\pl$-rank of an irreducible unitary representation of $G_\pl$ given in Definition 
\ref{dfnnurkk}, where $\pl\in\Places$.

\begin{thm}
\label{THMMaIn}
Let $(\pi,\ccH)$ be an irreducible unitary representation of $G_\A$ which occurs as a subrepresentation of $(\sfR,L^2(G_\Q\bls G_\A))$, and let the $\pi_\pl$, $\pl\in\Places$, denote the local components of $\pi$, as in Remark 
\ref{rmk-loccpnt}.
Then the $\pl$-rank of $\pi_\pl$ is independent of $\pl$.
\end{thm}
\begin{proof}
Let  $\ccH^\circ$ denote the G\aa rding space of $(\pi,\ccH)$ defined in Remark \ref{rmkAuFo}. For every  $\pl\in\Places$, let 
$d_\pl$ denote the
$\pl$-rank of $\pi_\pl$. If $d_\pl=r(\bfG)$ for every $\pl\in\Places$, then there is nothing to prove. Next assume that $d_\pl<r(\bfG)$ for some 
$\pl\in\Places$, and choose $\pl\in\Places$ 
such that $d_\pl$ has the smallest possible value.
%By Theorem \ref{mainsalduke} 
It suffices to prove that $d_{\pl_1}\leq d_\pl$ for every other  $\pl_1\in\Places$.

 Set $\sfR_{U_\pl}^{}:=\sfR\big|_{U_\pl}$. Lemma \ref{lemrkdsii} implies that $\sfR_{U_\pl}(\phi)f=0$ for every $\phi\in J_{d_\nu,\nu}$ and every $f\in \ccH$. 
Consider the vector space $W$ of complex-valued functions on $U_\Q\bls U_\A$ defined as
\[
W:=\left\{f\big|_{U_\Q\bls U_\A}\ :\ f\in \ccH^\circ\right\}.
\] Note that 
by Remark \ref{rmkAuFo}, 
elements of $\ccH^\circ$ are represented by  continuous maps 
$G_\Q\bls G_\A\to\C$,  and 
therefore their restriction to $U_\Q\bls U_\A$ is well-defined. 
Since  $U_\Q\bls U_\A$ is compact, elements of $W$ are bounded functions on $U_\Q\bls U_\A$, and in particular $W\subseteq L^2(U_\Q\bls U_\A)$.  
Let $\ccK\subseteq L^2(U_\Q\bls U_\A)$ denote the closure of $W$ inside  $L^2(U_\Q\bls U_\A)$.
Since $\ccH^\circ$ is $G_\A$-invariant, the space $W$ is $U_\A$-invariant. From $U_\A$-invariance of $W$ it follows that $\ccK$ is also $U_\A$-invariant. Consequently, we obtain a unitary representation 
$(\sigma,\ccK)$ 
of $U_\A$ on $\ccK$ obtained from the restriction of $\sfR'$ (see Lemma \ref{lemfbn}(ii)).   

\textbf{Step 1.} 
Let 
$\sfR'_{U_{\pl_1}}$ be as in
Lemma \ref{lemfbn}(ii).
In this step we prove that
\begin{equation}
\label{sfR'U}
\sfR'_{U_{\pl_1}}(\phi_1)w=0
\text{ for every }
\phi_1\in J_{d_\nu,\nu_1}
\text{ and  }w\in \ccK.
\end{equation}
Fix  
$\phi\in J_{d_\nu,\nu}$.
For any $f\in \ccH^\circ$,
Lemma \ref{lemfbn}(i) and 
continuity of the map 
$x\mapsto \int_{U_\nu}\phi(n)f(xn)dn$ (see Lemma \ref{lem:contPhi1Dc}) imply that
$
\int_{U_\nu}\phi(n)f(xn)dn=0
$ for every 
$x\in G_\Q\bls G_\A$. 
%(In fact in Steps 2 and 3 of the proof of Lemma \ref{lemfbn} it is shown that the map  is continuous.)
Thus from Lemma \ref{lemfbn}(ii) it follows that 
\begin{equation}
\label{eqsfR'}
\sfR'_{U_\pl}(\phi)w=0
\text{ for every }w\in W.
\end{equation} 
From Moore's result mentioned in Section \ref{KirillovA} it follows that $\sigma=\bigoplus_{\mu\in S}\rho_\mu$, where 
$S\subseteq \g u_\Q^*$ and  the representations 
$\rho_\mu=\otimes_{\pl\in\Places}^{}\rho_{\mu_\pl}$ are defined in 
\eqref{rhomRTP}. Now fix $\mu\in S$. Since $\rho_\mu\big|_{U_\pl}$ is a direct sum of countably many copies of $\rho_{\mu_\pl}$, 
from \eqref{eqsfR'} and Lemma \ref{lemrkdsii} it follows that $\rho_{\mu_\pl}\in \widehat{U}_\pl[d_\pl]$. 
Consequently, by Lemma \ref{orbitdim} we obtain  
\[
\rho_{\mu_{\pl_1}}\in\widehat{U}_{\pl_1}[d_\pl]
\text{ for every }\pl_1
\in\Places
.\]
Therefore  Lemma \ref{lemrkdsii} 
implies \eqref{sfR'U}.

\textbf{Step 2.}
Let
$G^{\pl_1}$ be 
defined as in Section \eqref{Gupv}.
For every 
$g\in G^{\nu_1}$, the map 
\[
\bfU(\Q_{\nu_1})\to G_\A\ ,\ n\mapsto g\mathsf{s}_{\nu_1}(n)g^{-1}
\]
is a splitting section. From the uniqueness of this section (see Section \ref{Sec-nurknk})
it follows that $ng=gn$ for every $n\in U_{\pl_1}$.

\textbf{Step 3.} By the weak approximation property, $\bfG(\Q)\bfG^{\pl_1}$ is dense in $\bfG(\A)$. 
It follows that $G_\Q G^{\pl_1}F$ is a dense subset of $G_\A$, where $F\sseq G_\A$ denotes the kernel of the central extension \eqref{exactseqA}.

\textbf{Step  4.} 
In this step we prove that
\begin{equation}
\label{vanRphi}
\sfR_{U_{\pl_1}}(\phi_1)f=0\text{ for every  }f\in\ccH^\circ\text{ and }\phi_1\in J_{d_\pl,\pl_1}.
\end{equation}

Recall that by Remark \ref{dfnabus}
and Remark \ref{rmkAuFo}, we can represent every element of $\ccH^\circ$ by a unique continuous map $G_\A\to\C$. %By a slight abuse of notation, we denote the latter map by $f$ as well. 
Furthermore, Lemma \ref{lem:contPhi1Dc} and 
$G_\A$-invariance of $\ccH^\circ$ imply that the map $x\mapsto \int_{U_{\nu_1}}\phi_1(n)f(xng)dn$ 
is continuous for every $f\in\ccH^\circ$ and $g\in G_\A$.
Thus by Lemma \ref{lemfbn}(ii) and Step 1
  we obtain
\begin{equation}
\label{eqphfxng}
\int_{U_{\pl_1}}\phi_1(n)f(xng)dn=0\text{ for every  }\,
\phi_1\in J_{d_\nu,\nu_1},\ 
f\in\ccH^\circ,\ g\in G_\A,\text{ and }x\in U_\Q\bls U_\A.
\end{equation}
%Now let $P_\A$ denote the set of $\A$-points of the parabolic subgroup $\bfP$. 
Setting $x=U_\Q$ (the identity coset in $U_\A$) in  
\eqref{eqphfxng}, and using Step 2, we obtain that
%\left(\sfR_{U_{\pl_1}}(\phi_1)f\right)(g)=
\begin{equation}
\label{eqphfxng1}
\int_{U_{\pl_1}}\phi_1(n)f(gn)dn=0\text{ for every  }\,
\phi_1\in J_{d_\pl,\pl_1},\ 
g\in G^{\pl_1},
\text{ and }
f\in\ccH^\circ.
\end{equation}
%Our next goal is to prove that the left hand side of \eqref{eqphfxng1} vanishes for every $g\in G_\A$. 
Note that in \eqref{eqphfxng1} we consider $f$ as a map $G_\A\to\C$ (see Remark \ref{dfnabus}).

By Lemma \ref{lemfbn}(i),
continuity of the left hand side of 
\eqref{eqphfxng1} as a function of $g$
(see Lemma \ref{lem:contPhi1Dc}), and Step 3, in order to complete the proof of \eqref{vanRphi} it suffices to prove that
the vanishing condition 
\eqref{eqphfxng1}
%$\sfR_{U_{\pl_1}}(\phi_1)f$, 
holds  for every $g\in G_\Q G^{\nu_1}F$.
From left $G_\Q$-invariance of
the left hand side of 
\eqref{eqphfxng1} it follows that 
this vanishing condition also holds for every $g\in G_\Q G^{\pl_1}$. 
Finally, Schur's Lemma implies that for every $z\in F$, the action of $\pi(z)$ 
on $\ccH$ is by a scalar $\gamma(z)\in\C$, 
and therefore for every $g\in G_\Q G^{\pl_1}$,
\[
\int_{U_{\pl_1}}\phi_1(n)f(gzn)dn
=\int_{U_{\pl_1}}\phi_1(n)f(gnz)dn
=\gamma(z)\int_{U_{\pl_1}}\phi_1(n)f(gn)dn=0
.
\]
%For $p\in P_\A$ and $g\in G_\A$,
%\begin{align}
%\label{eq-prq}
%\int_{U_{\pl_1}}\phi(n)f(png)dn&=
%\int_{U_{\pl_1}}\phi(n)f(pnp^{-1}pg)dn
%=\delta(p)\int_{U_{\pl_1}}\phi^p(\tilde n)
%f(\tilde{n}pg)d\tilde n,
%\end{align}
%where $\phi^p(\tilde n):=\phi(p^{-1}\tilde np)$ and 
%$\delta:P_\A\to\R^+$ is the (uniquely defined) group homomorphism satisfying 
%$\delta(p)d(pnp^{-1})=dn$.
%Because of \eqref{eqphfxng} and Remark \ref{rmk-pinv},
%the right hand side of \eqref{eq-prq} vanishes. Since %$f$ is left-$G_\Q$-invariant, we obtain 
%\[
%\int_{U_{\pl_1}}\phi(n)f(qpng)dn=0\text{ for every }q\in G_\Q,\ p\in P_\A,\ \text{and }
%g\in G_\A.
%\]
%From Lemma \ref{strApprox} we know that $G_\Q P_\A$ is dense in $G_\A$, and from Steps 2 and 3 in the proof of Lemma \ref{lemfbn} it follows that the map $G_\A\to \C$, $x\mapsto \int_{U_{\pl_1}}
%\phi(n)f(xng)dn$ is continuous. Therefore the latter map vanishes for every $x\in G_\A$. 
%Consequently,  Lemma
%\ref{lemfbn}(i) implies that 

\textbf{Step 5.} Since 
$\ccH^\circ$ is dense in $\ccH$, the assertion
\eqref{vanRphi} of Step 4 holds for every $f\in\ccH$ as well. Now Lemma \ref{lemrkdsii} and Theorem 
\ref{mainsalduke} imply that   
the ${\pl_1}$-rank of $\pi_{\pl_1}$ is at most $d_\pl$. 
\end{proof}

\begin{rmk}
For several groups $\bfG$ it is known (see \cite{GanSavin}, \cite{SalManus}) that  
for every $\pl\in\Places$, an irreducible unitary representation of $G_{\pl}$ is minimal if and only if its $\nu$-rank is equal to one.
For such $\bfG$, Theorem \ref{THMMaIn} implies that   if at least one local component of an automorphic representation of $G_\A$ is a minimal representation, then all of its local components are minimal.
\end{rmk}

\bibliographystyle{plain}
\bibliography{lowrank}

\end{document}